\documentclass{amsart}
\usepackage{amsmath}
\usepackage{amssymb}
\usepackage{amsthm}
\usepackage{tikz}
\usepackage{url}
\newtheorem{thm}{Theorem}
\newtheorem{lem}[thm]{Lemma}
\theoremstyle{definition}
\newtheorem{problem}[thm]{Problem}
\newtheorem{rem}[thm]{Remark}
\newcommand{\AREA}{\operatorname{area}}
\newcommand{\CONV}{\operatorname{conv}}
\newcommand{\INT}{\operatorname{int}}
\newcommand{\BD}{\operatorname{bd}}
\newcommand{\AFF}{\operatorname{Aff}}

\begin{document}

\title{Self-affinity of discs under glass-cut dissections}
\author{Christian Richter}
\address{Institute for Mathematics, Friedrich Schiller University, 07737 Jena, Germany}
\date{\today}

\begin{abstract}
A topological disc is called $n$-self-affine if it has a dissection into $n$ affine images of itself. It is called $n$-gc-self-affine if the dissection is obtained by successive glass-cuts, which are cuts along segments splitting one disc into two. For every $n \ge 2$, we characterize all $n$-gc-self-affine discs. All such discs turn out to be either triangles or convex quadrangles. All triangles and trapezoids are $n$-gc-self-affine for every $n$. Non-trapezoidal quadrangles are not $n$-gc-self-affine for even $n$. They are $n$-gc-self-affine for every odd $n \ge 7$, and they are $n$-gc-self-affine for $n=5$ if they aren't affine kites. Only four one-parameter families of quadrangles turn out to be $3$-gc-self-affine.

In addition, we show that every convex quadrangle is $n$-self-affine for all $n \ge 5$.
\end{abstract}

\subjclass[2010]{52C20 (primary); 05B45; 51N10; 52B45 (secondary).}
\keywords{Topological disc; self-affine; glass-cut dissection; guillotine cut}

\maketitle


\section{Background and main results}

A \emph{topological disc} $D$ in the Euclidean plane is an image of a closed circular disc under an affine transformation of the plane. The disc $D$ is called \emph{$n$-self-affine} if it has a dissection into discs $D_1,\ldots,D_n$, called \emph{pieces}, that are all affine images of $D$. Here we speak of a \emph{dissection} if $D$ is the union of all pieces and any two pieces have disjoint interiors. We call $D$ \emph{self-affine} if it is $n$-self-affine for some integer $n \ge 2$. The concept of self-affinity generalizes self-similarity and is prominent in fractal geometry, where the affine transformations are supposed to be contractions \cite[Section 9.4]{falconer1990}. But it is also fruitful in the elementary geometry of polygons, see e.g. \cite{blechschmidt_richter2015,hertel2000,hertel_richter2010,richter2012}. Figure~\ref{fig:1} presents some examples. The non-convex ones are adopted from \cite{golomb1964}. 
Self-affine convex polygons have at most five vertices as follows from \cite[Theorem~5]{bernheim_motzkin1949}, see also \cite[Satz~1]{hertel2000}. Self-affinity of triangles is trivial. All convex quadrangles are $5$-self-affine. This goes back to Attila P\'or, see \cite[Proposition~1]{hertel_richter2010}. There exist self-affine convex pentagons \cite[Proposition~4]{hertel_richter2010}, but the regular pentagon and, more generally, pentagons whose inner angles are close to $\frac{3\pi}{5}=108^\circ$ are not self-affine \cite[Proposition~3]{hertel_richter2010}, \cite[Theorem~1]{blechschmidt_richter2015}.

\begin{figure}
\begin{center}
\begin{tikzpicture}[xscale=.067, yscale=.058]

\draw
  (-41.0,0)--(-66.0,0)--(-53.8,25.6)--cycle
	(-41,0)--(-53.8,8.5)--(-66,0)
	(-53.8,8.5)--(-53.8,25.6)
	;

\draw
  (-10.0,0)--(-31.0,0)--(-31.0,24.0)--(-20.5,24.0)--(-20.5,12.0)--(-10.0,12.0)--cycle
	(-15.25,12.0)--(-15.25,6.0)--(-25.75,6.0)--(-25.75,18.0)--(-20.5,18.0)
	(-20.5,0)--(-20.5,6.0)
	(-31.0,12.0)--(-25.75,12.0)
	;

\draw
  (12.8,0)--(25.6,0)--(19.2,12.8)--(12.8,12.8)--(11.2,9.6)--(12.0,8.0)--(12.8,8.0)--(13.0,8.4)--(12.9,8.6)
  (6.4, 12.8) -- (0, 0) -- (12.8, 0) -- (16.0, 6.4) -- (14.4, 9.6) -- (12.8, 9.6) -- (12.4, 8.8) -- (12.6, 8.4) -- (12.8, 8.4)
  (19.2, 12.8) -- (12.8, 25.6) -- (6.4, 12.8) -- (9.6, 6.4) -- (12.8, 6.4) -- (13.6, 8.0) -- (13.2, 8.8) -- (12.8, 8.8) -- (12.7, 8.6)
  (25.6, 0) -- (32.0, 12.8) -- (28.8, 19.2) -- (25.6, 19.2) -- (24.8, 17.6) -- (25.2, 16.8) -- (25.6, 16.8) -- (25.7, 17.0)
  (19.2, 12.8) -- (25.6, 12.8) -- (27.2, 16.0) -- (26.4, 17.6) -- (25.6, 17.6) -- (25.4, 17.2) -- (25.5, 17.0) 
  ;
	
\draw
  (40,0)--(65,0)--(52.5,25.6)-- cycle
	(52.5,25.6)--(52.5,0)
  (40,0)--(52.5,8.5)--(65,0)
	;
	
\draw
  (75,0)--(105,0)--(95,25.6)--(75,25.6)--cycle
	(81,0)--(79,25.6)
	(87,0)--(83,25.6)
	(90,12.8) node {$\cdots$}
	(99,0)--(91,25.6)
	;
	
\end{tikzpicture}
\end{center}

\caption{Realizations of self-affinities and gc-self-affinities \label{fig:1}}
\end{figure}
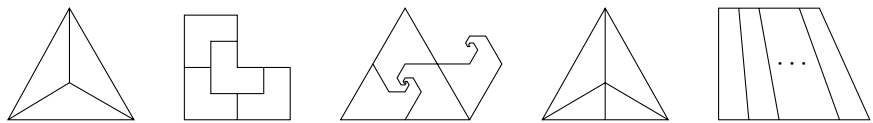

In the present paper our main focus is on so-called glass-cut dissections. A \emph{glass-cut} (also guillotine cut) dissects a disc $D$ along a line segment into two discs. That process can be repeated finitely many times with some of the resulting pieces. The outcome is called a \emph{glass-cut dissection} (or \emph{gc-dissection} for short) of $D$. Accordingly, we obtain the concepts of \emph{$n$-gc-self-affine} and \emph{gc-self-affine} discs. Glass-cut dissections are less flexible than general ones, but are more accessible to algorithmic approaches, see e.g. \cite{czyzowicz_et_al2007, kranakis_et_al2000, martini_soltan1998, puchinger_raidl2007, seiden_woeginger2005} for applications. Only the last two dissections in Figure~\ref{fig:1} are based on glass-cuts.

Our main result is the following characterization of all $n$-gc-self-affine discs for every $n=2,3,\ldots$

\begin{thm}\label{thm:main}
\begin{enumerate}
\item[(i)] 
Every gc-self-affine topological disc is a triangle or a convex quadrangle.
\item[(ii)] 
Every triangle is $n$-gc-self-affine for all $n=2,3,\ldots$
\item[(iii)] 
Let $n \in \mathbb{Z}$, $n \ge 2$, be even. Then a convex quadrangle is $n$-gc-self-affine if and only if it is a trapezoid.
\item[(iv)] 
Let $n \in \mathbb{Z}$, $n \ge 7$, be odd. Then every convex quadrangle is $n$-gc-self-affine.
\item[(v)] 
A convex quadrangle is not $5$-gc-self-affine if and only if it is an affine image of a kite, but no parallelogram.
\item[(vi)] 
The $3$-gc-self-affine convex quadrangles are given by four one-parameter families, see Theorem~\ref{thm:n=3} for details.
\end{enumerate}
\end{thm}

Theorem~\ref{thm:main} shows in particular that gc-self-affinity of every non-trapezoidal quadrangle $Q$ is non-trivial in so far as there is no number $n_0$ such that $Q$ is $n$-gc-self-affine for all $n \ge n_0$. Let us point out that the situation is different for self-affinity based on general dissections.

\begin{thm}\label{thm:self-affine}
Every convex quadrangle is $n$-self-affine for every $n \in \mathbb{Z}$, $n \ge 5$.
\end{thm}

Since every dissection of a convex quadrangle into two quadrangles is done by a glass-cut, the only $2$-self-affine convex quadrangles are trapezoids by Theorem~\ref{thm:main} (see also \cite[Satz 3]{hertel2000}). In view of Theorem~\ref{thm:self-affine}, the following question arises.

\begin{problem}
What convex quadrangles are $n$-self-affine for $n=3,4$ ?
\end{problem}

First systematic considerations of that problem can be found in \cite{hertel2000, zimmermann2023}.

The remainder of the present paper is organized as follows. In Section~\ref{sec:convex} we show that qc-self-affine topological discs are necessarily convex. The short Section~\ref{sec:triangles_quadrangles} proves Theorem~\ref{thm:main}(i). Theorem~\ref{thm:main}(ii) is trivial.
In Section~\ref{sec:parametrization} we start the discussion of quadrangles by introducing a appropriate parametrization of affine types of quadrangles. In Section~\ref{sec:quadrangles} we analyse the parameters of a quadrangle that is composed by glueing together two given quadrangles along a common side. Section~\ref{sec:trapezoids} is devoted to the proof of Theorem~\ref{thm:main}(iii), Section~\ref{sec:n>=5} concerns Theorem~\ref{thm:main}(iv) and (v), and Section~\ref{sec:n=3} gives Theorem~\ref{thm:main}(vi). Finally, Section~\ref{sec:self-affine} proves Theorem~\ref{thm:self-affine}.

We use the following notations. Open and closed intervals in $\mathbb R$ are denoted by $(\xi,\eta)$ and $[\xi,\eta]$, respectively. The line segment in $\mathbb R^2$ with endpoints $x$ and $y$ is denoted by $xy$, its length by $|xy|$. The straight line through $x$ and $y$ is $l(xy)$. We write $B(x_0,r)$ for the closed circular disc (or ball) $\{x \in \mathbb R^2: |xx_0| \le r\}$ of radius $r>0$ centred at $x_0 \in \mathbb R^2$. Interior, boundary, convex hull and area of a plane set $X$ are $\INT(X)$, $\BD(X)$, $\CONV(X)$ and $\AREA(X)$, respectively.


\section{All gc-self-affine discs are convex\label{sec:convex}}

Let the disc $D$ be $n$-gc-self-affine based on a dissection into affine images $\varphi_1(D)$, $\ldots$, $\varphi_n(D)$. For every integer $k \ge 1$, the affine images $\varphi_{i_1} \circ \ldots \circ \varphi_{i_k}(D)$, $i_1,\ldots,i_k \in \{1,\ldots,n\}$, form an (iterated) $n^k$-gc-self-affinity of $D$. Every sequence $(i_k)_{k=1}^\infty \subseteq \{1,\ldots,n\}$ gives rise to a decreasing sequence $(\varphi_{i_1} \circ \ldots \circ \varphi_{i_k}(D))_{k=1}^n$ of compact sets, whose \emph{limit set} in the Hausdorff metric is $S=\bigcap_{k=1}^\infty \varphi_{i_1} \circ \ldots \circ \varphi_{i_k}(D) \subseteq D$ (see \cite[Lemma 1.8.2]{schneider2014}). Since the determinants of (the linear maps associated to) $\varphi_{i_1} \circ \ldots \circ \varphi_{i_k}$, $k=1,2,\ldots$, tend to zero, $S$ is either a singleton or a non-degenerate line segment. In the latter case we speak of a \emph{limit segment}. 

\begin{lem}\label{lem:limit_sets}
The following are satisfied for every qc-self-affine disc $D$.
\begin{enumerate}
\item[(i)] 
Every $x \in D$ belongs to some limit set.
\item[(ii)] 
Two limit segments $S_1$ and $S_2$ do not cross in the sense that $S_1 \cap S_2$ is a singleton in the relative interiors of both $S_1$ and $S_2$.
\item[(iii)] 
If some limit set $S$ satisfies $S \subseteq \INT(D)$, then $D$ is convex.
\end{enumerate}
\end{lem}

\begin{proof}
For (i), one easily checks that there is a sequence $(i_k)_{k=1}^\infty$ such that $x \in \varphi_{i_1} \circ \ldots \circ \varphi_{i_k}(D)$ for $k=1,2,\ldots$ 

For (ii), assume that the Hausdorff limits $S_1=\bigcap_{k=1}^\infty \varphi_{i_1} \circ\ldots\circ\varphi_{i_k}(D)$ and $S_2=\bigcap_{k=1}^\infty \varphi_{j_1} \circ\ldots\circ\varphi_{j_k}(D)$ cross. Hence, for all $k$, the pieces $\varphi_{i_1} \circ\ldots\circ\varphi_{i_k}(D)$ and $\varphi_{j_1} \circ\ldots\circ\varphi_{j_k}(D)$ of the $k$-th iterated dissection share some inner points. Then $(i_k)_{k=1}^\infty=(j_k)_{k=1}^\infty$ and in turn $S_1=S_2$, a contradiction.

For (iii), note that, if $S=\bigcap_{k=1}^\infty \varphi_{i_1} \circ \ldots \circ \varphi_{i_k}(D) \subseteq \INT(D)$, then $D_0=\varphi_{i_1} \circ \ldots \circ \varphi_{i_{k_0}}(D) \subseteq \INT(D)$ for some sufficiently large $k_0$. Since the piece $D_0$ of the $k_0$-th iterated gc-dissection is in $\INT(D)$, its boundary is formed by finitely many line segments and its inner angles are smaller than $\pi$. Hence $D_0=\varphi_{i_1} \circ \ldots \circ \varphi_{i_{k_0}}(D)$ is a convex polygon, and so is $D$. 
\end{proof}

Unfortunately, there are gc-self-affinities of some discs $D$ that do not give rise to limit sets completely contained in the interior of $D$, see e.g. the right-most example in Figure~\ref{fig:1}. Therefore Lemma~\ref{lem:limit_sets}(iii) is not enough to show convexity of every gc-self-affine disc. A deeper analysis is needed.

\begin{lem}\label{lem:segments1}
Let $D$ be a non-convex gc-self-affine topological disc. Then there exists a limit segment $S=ab$ such that $S \cap \BD(D) \subseteq \{a,b\}$.
\end{lem}

\begin{proof}
We divide the proof into two parts.

\emph{Claim 1. There exist $r > 0$ and uncountably many limit segments, all placed on mutually different straight lines and being of length at least $r$, and each meeting $\BD(D)$.}

We pick some $x \in \INT(D)$ and $r > 0$ such that $B(x,2r) \subseteq \INT(D)$, see the left-hand part of Figure~\ref{fig:segment1}.
\begin{figure}
\begin{center}
\begin{tikzpicture}[scale=.9]

\draw[thick]
  (-1,-.5) arc (-180:0:2 and 1)
	(-1,-.5)--(-1,.5)
	(3,-.5)--(3,1.5)
	(-1,.5) arc (360:180:.5)
	(3.5,1.625) arc (180:360:.25)
	(4,1.625)--(4,2)
	(4,2) arc (0:180:3 and 1.2)
	(-2,.5)--(-2,2)
	;

\draw
	(1,1.2) circle (.55)
	(1,1.2) circle (1.1)
	;

\draw[line width=1.3]
	(-1.5,.375)--(4,1.75)
	(-2,1.2)--(4,2)
	(1,.8)--(2.5,-1.15)
  ;

\draw[line width=4]
  (3,1.5)--(3.5,1.625)
	;

\fill
  (-1.8,2.9) node {\large $D$}
  (-1,.5) circle (.08) 
  (3.5,1.625) circle (.08)
  (3,1.5) circle (.08)
  (4,1.75) circle (.08)
	(-2,1.2) circle (.08)
	(4,2) circle (.08)
	(2.5,-1.15) circle (.08)
	(1,1.94) node {\footnotesize $B(x,r)$}
	(1,2.55) node {$B(x,2r)$}
	;

\draw[thick]
  (9,.8) circle (2.25)
	;
	
\draw[line width=1.3]
  (6.55,2.08)--(6.65,2.09) (6.75,2.1)--(11.25,2.55) (11.35,2.56)--(11.45,2.57)
	(6.55,.77)--(6.65,.785) (6.75,.8)--(11.25,1.475) (11.35,1.49)--(11.45,1.505)
	(6.55,1.7)--(6.65,1.7) (6.75,1.7)--(11.25,1.7) (11.35,1.7)--(11.45,1.7)
	(6.55,-.08)--(6.65,-.09) (6.75,-.1)--(11.25,-.55) (11.35,-.56)--(11.45,-.57)
	;

\fill
  (6.8,2.86) node {\large $B\left(x_0,\frac{\varepsilon}{2}\right)$}
  (9,.8) circle (.04) node[below] {\large $x_0$}
  (9.8,2.405) circle (.04) node[above] {$x_1$}
  (9.3,-.35) circle (.04) node[above] {$x_2$}
  (8.7,1.7) circle (.04) node[below] {$x_3$}
  (8.6,1.08) circle (.04) node[below] {$x_4$}
	(8.3,2.52) node {$S_1$}
	(8.3,-.53) node {$S_2$}
	(7.7,1.43) node {$S_3$}
	(10.5,1.07) node {$S_4$}
	;

\end{tikzpicture}
\end{center}
\caption{Proof of Lemma~\ref{lem:segments1}\label{fig:segment1}}
\end{figure}
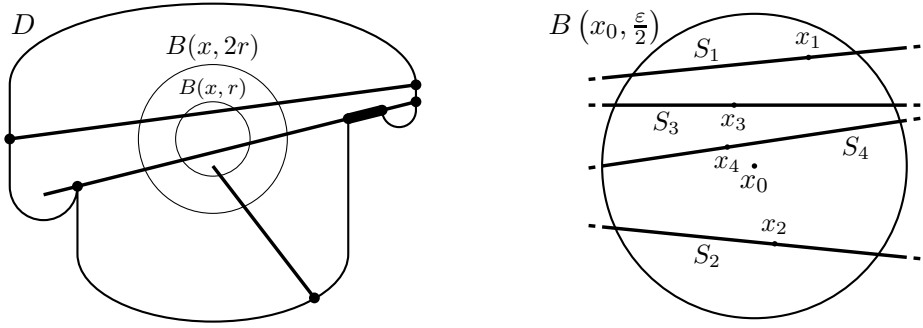
By Lemma~\ref{lem:limit_sets}(i) and (iii), every point from $B(x,r)$ belongs to some limit segment that meets $\BD(D)$. By the triangle inequality, their lengths are larger than $r$. Since countably many straight lines cannot cover $B(x,r)$, there are uncountably many such limit segments placed on different lines. Figure~\ref{fig:segment1} illustrates three examples and their intersections with $\BD(D)$. 

\emph{Claim 2. At most countably many limit segments $ab$ found in Claim 1 satisfy $ab \cap \BD(D) \not \subseteq \{a,b\}$.} 

We shall prove Claim 2 by showing that, for every $\varepsilon >0$, at most finitely many of the segments $S=ab$ described in Claim 1
 contain some $x \in S \cap \BD(D)$ such that $\min\{|ax|,|bx|\} > \varepsilon$.

Suppose that, contrary to the last assertion, there are infinitely many limit segments $S_i=a_ib_i$ and points $x_i \in (S_i \cap \BD(D)) \setminus (B(a_i,\varepsilon) \cup B(b_i,\varepsilon))$, $i=1,2,\ldots$ Since $\BD(D)$ is compact, we can assume that $(x_i)_{i=1}^\infty$ converges to some $x_0 \in \BD(D)$ and that $(x_i)_{i=1}^\infty \subseteq B\left(x_0,\frac{\varepsilon}{2}\right)$, see the right-hand part of Figure~\ref{fig:segment1}. By the triangle inequality, the end-points of each $S_i$ are outside $B\left(x_0,\frac{\varepsilon}{2}\right)$. By Lemma~\ref{lem:limit_sets}(ii), distinct segments $S_i$ do not meet in $B\left(x_0,\frac{\varepsilon}{2}\right)$.

Since $x_0 \in \BD(D)$, $\BD(D) \cap B\left(x_0,\frac{\varepsilon}{2}\right)$ contains a connected component $\Gamma$ of $\BD(D)$ that contains $x_0$ in its relative interior. Since $(x_i)_{i=1}^\infty \subseteq \BD(D)$ and $\lim_{i \to \infty} x_i=x_0$, we get $(x_i)_{i=i_0}^\infty \subseteq \Gamma$ for some $i_0$. But then the arc $\Gamma \subseteq B\left(x_0,\frac{\varepsilon}{2}\right)$ has to connect all the points $x_i$, $i=i_0,i_0+1,\ldots$, without crossing any of the segments $S_i$, because $S_i \subseteq D$ and $\Gamma \subseteq \BD(D)$. This is impossible. The proof of Claim 2 and of Lemma~\ref{lem:segments1} is complete.
\end{proof}

\begin{lem}\label{lem:segments2}
Let $D$ be a gc-self-affine topological disc having a limit segment $S=ab$ such that $S \cap \BD(D) \subseteq \{a,b\}$. Then, for every $\varepsilon \in \left(0,\frac{1}{2}\right]$, there is an equiaffine map $\alpha_\varepsilon: \mathbb R^2 \to \mathbb R^2$ such that
\begin{enumerate}
\item[(i)]
$[-1,1] \times \{0\} \subseteq \alpha_\varepsilon(D) \subseteq [-1-\varepsilon,1+\varepsilon] \times \mathbb R$,
\item[(ii)]
$\alpha_\varepsilon(D) \cap ([-1+\varepsilon,1-\varepsilon] \times \mathbb R)$ is convex,
\item[(iii)]
$\alpha_\varepsilon(D) \cap ([-1+\varepsilon,1-\varepsilon] \times \mathbb R) \subseteq [-1+\varepsilon,1-\varepsilon] \times [-2\AREA(D),2\AREA(D)]$.
\end{enumerate}
\end{lem}

\begin{proof}
W.l.o.g., i.e., after some equiaffine transformation, we have $S=[-1,1] \times \{0\}$, i.e., $a=\genfrac{(}{)}{0pt}{1}{-1}{0}$ and $b=\genfrac{(}{)}{0pt}{1}{1}{0}$.
Since $S$ is the Hausdorff limit of a decreasing sequence of pieces in certain gc-self-affinities of $D$ and since $S \setminus \{\genfrac{(}{)}{0pt}{1}{-1}{0},\genfrac{(}{)}{0pt}{1}{1}{0}\} \subseteq \INT(D)$, there exist some $\delta=\delta(\varepsilon) \in (0,\varepsilon]$ and an affine transformation $\beta_\varepsilon$ such that
\begin{eqnarray}
&[-1,1]\times \{0\} \subseteq \beta_\varepsilon(D) \subseteq [-1-\varepsilon,1+\varepsilon] \times \mathbb R,&\label{eq:p6-3}\\
&\beta_\varepsilon(D) \cap ([-1+\varepsilon,1-\varepsilon] \times \mathbb R) \subseteq [-1+\varepsilon,1-\varepsilon]\cap [-\delta,\delta] \subseteq \INT(D).&\label{eq:p6-4}
\end{eqnarray}
Since the piece $\beta_\varepsilon(D)$ belongs to a gc-dissection of $D$, $\BD(\beta_\varepsilon(D)) \cap \INT(D)$ is a finite union of polygonal arcs with inner angles smaller than $\pi$. Hence \eqref{eq:p6-4} implies
\begin{equation}
\beta_\varepsilon(D) \cap ([-1+\varepsilon,1-\varepsilon] \times \mathbb R) \text{ is convex.}\label{eq:p6-5}
\end{equation}
We define the equiaffine map $\alpha_\varepsilon$ as $\alpha_\varepsilon=\gamma \circ \beta_\varepsilon$ with $\gamma\genfrac{(}{)}{0pt}{1}{\xi_1}{\xi_2}=\big(\genfrac{}{}{0pt}{1}{\xi_1}{\frac{1}{\det(\beta_\varepsilon)}\xi_2}\big)$. Then \eqref{eq:p6-3} implies (i) and \eqref{eq:p6-5} yields (ii).
Finally, since $\varepsilon \le \frac{1}{2}$, (i) gives
\[
\textstyle\left[-\frac{1}{2},\frac{1}{2}\right] \times \{0\} \subseteq \alpha_\varepsilon(D) \cap ([-1+\varepsilon,1-\varepsilon] \times \mathbb R).
\]
By (ii) and $\AREA(\alpha_\varepsilon(D))=\AREA(D)$, this shows (iii).
\end{proof}

Every equiaffine map $\alpha: \mathbb R^2 \to \mathbb R^2$ has a unique representation $\alpha\genfrac{(}{)}{0pt}{1}{\xi_1}{\xi_2}=\genfrac{(}{)}{0pt}{1}{\alpha_{11}\xi_1+\alpha_{12}\xi_2+\alpha_{13}}{\alpha_{21}\xi_1+\alpha_{22}\xi_2+\alpha_{23}}$ with $(\alpha_{11},\ldots,\alpha_{23}) \in \mathbb R^6$. This way the set $\AFF_2^1$ of all equiaffine maps corresponds to a closed subset of $\mathbb R^6$, and the natural topology of $\mathbb R^6$ gives rise to a respective topology on $\AFF_2^1$. The following lemma is routine, based on elementary topological properties of $\mathbb R^n$ and the localization of the Hausdorff convergence in the sense of \cite[Theorem 1.8.8]{schneider2014}.

\begin{lem}\label{lem:routine}
\begin{enumerate}
\item[(i)]
For every $x_0 \in \mathbb R^2$, every $r > 0$ and every compact set $C \subseteq \mathbb R^2$, $\{\alpha \in \AFF_2^1: \alpha(B(x_0,r)) \subseteq C\}$ is compact.
\item[(ii)]
Suppose that $\alpha_0,\alpha_1,\ldots \in \AFF_2^1$ satisfy $\lim_{i \to \infty} \alpha_i= \alpha_0$. Then, for every non-empty compact set $C \subseteq \mathbb R^2$, $\lim_{i \to \infty} \alpha_i(C)=\alpha_0(C)$ in the Hausdorff metric.
\end{enumerate}
\end{lem}
 
\begin{lem}\label{lem:localization}
Let $D$ be a topological disc. Then there exist a disc $B(x_0,r) \subseteq D$, $x_0 \in D$, $r > 0$, and $\mu \in (0,1)$ such that, for every affine functional $\beta: \mathbb R^2 \to \mathbb R$,
\begin{align*}
[\min \beta(B& (x_0,r)), \max \beta(B(x_0,r))]\\
 \subseteq &
\textstyle\left[(1-\mu)\min \beta(D)+\mu\max \beta(D), (1-\mu)\max \beta(D)+\mu\min \beta(D)\right].
\end{align*}
\end{lem}

\begin{proof}
We fix $x_0 \in \INT(D)$ and $R > r > 0$ such that 
\[
B(x_0,r) \subseteq B(x_0,2r) \subseteq D \subseteq B(x_0,R).
\]
Since $\beta$ is affine, we have
\begin{eqnarray}
&\textstyle\frac{3}{2}\min\beta(B(x_0,r))-\frac{1}{2}\max \beta(B(x_0,r))=\min \beta(B(x_0,2r)),&\label{eq:p8-e4}\\
&\textstyle\frac{r-R}{2r}\min\beta(B(x_0,r))+\frac{r+R}{2r}\max \beta(B(x_0,r))=\max \beta(B(x_0,R)).&\label{eq:p8-e5}
\end{eqnarray}
By \eqref{eq:p8-e4} and $\min\beta(B(x_0,2r)) \ge \min \beta(D)$,
\[
3(r+R) \min\beta(B(x_0,r))- (r+R)\max \beta(B(x_0,r)) \ge 2(r+R) \min\beta(D).
\]
Similarly, by \eqref{eq:p8-e5} and $\max\beta(B(x_0,R)) \ge \max\beta(D)$,
\[
(r-R)\min\beta(B(x_0,r))+(r+R)\max \beta(B(x_0,r)) \ge 2r \max\beta(D).
\]
Adding the last two inequalities and dividing by $4r+2R$, we get
\begin{equation}\label{eq:p8-e6}
\textstyle
\min\beta(B(x_0,r)) \ge \frac{r+R}{2r+R} \min \beta(D)+\frac{r}{2r+R} \max\beta(D).
\end{equation}
In the same way we arrive at
\begin{equation}\label{eq:p8-e7}
\textstyle
\max\beta(B(x_0,r)) \le \frac{r+R}{2r+R} \max \beta(D)+\frac{r}{2r+R} \min\beta(D).
\end{equation}
Inequalities \eqref{eq:p8-e6} and \eqref{eq:p8-e7} yield the claim of Lemma~\ref{lem:localization} with $\mu=\frac{r}{2r+R}$.
\end{proof}

\begin{lem}\label{lem:convex}
If a topological disc $D$ is gc-self-affine such that there is a limit segment $S=ab$ with $S \cap\BD(D) \subseteq \{a,b\}$, then $D$ is convex.
\end{lem}

\begin{proof}
We use the equiaffine maps $\alpha_\varepsilon$, $0 < \varepsilon \le \frac{1}{2}$, obtained from Lemma~\ref{lem:segments2} and the circular disc $B(x_0,r) \subseteq D$ and the real $\mu$ given by Lemma~\ref{lem:localization}. 

Let $\varepsilon_0=\min\left\{\frac{2\mu}{2-\mu},\frac{1}{2}\right\}$. First we show that, for all $\varepsilon \in (0,\varepsilon_0]$,
\begin{equation}\label{eq:pf9-1}
\alpha_\varepsilon(B(x_0,r)) \subseteq [-1,1] \times [-2\AREA(D),2\AREA(D)].
\end{equation}
For that, let $\varepsilon \in (0,\varepsilon_0]$ be fixed and let $\pi:\mathbb R^2 \to \mathbb R$ denote the projection onto the first coordinate. Then Lemma~\ref{lem:localization} and Lemma~\ref{lem:segments2}(i) give
\begin{align*}
\min \pi \circ \alpha_\varepsilon(B(x_0,r)) &\ge (1-\mu)\cdot \min \pi \circ \alpha_\varepsilon(D)+\mu \cdot \max \pi \circ \alpha_\varepsilon(D)\\
&\ge (1-\mu)(-1-\varepsilon)+\mu \cdot 1\\
&= -1+(2\mu-(1-\mu)\varepsilon)\\
& \ge -1+\varepsilon,
\end{align*}
where the last estimate comes from $\varepsilon \le \varepsilon_0 \le \frac{2\mu}{2-\mu}$. In the same way we get
\[
\max \pi \circ \alpha_\varepsilon(B(x_0,r)) \le 1-\varepsilon.
\]
Consequently, the first coordinates of all points from $\alpha_\varepsilon(B(x_0,r))$ are in $[-1+\varepsilon,1-\varepsilon]$. Hence
\[
\alpha_\varepsilon(B(x_0,r)) \subseteq ([-1+\varepsilon,1-\varepsilon]\times\mathbb R) \cap \alpha_\varepsilon(D),
\]
and Lemma~\ref{lem:segments2}(iii) yields \eqref{eq:pf9-1}.

For $\varepsilon_n=\frac{1}{n}$, \eqref{eq:pf9-1} gives
\[
\alpha_{\frac{1}{n}}(B(x_0,r)) \subseteq [-1,1] \times [-2\AREA(D),2\AREA(D)],
\]
provided $n \ge n_0= \lceil\frac{1}{\varepsilon_0}\rceil$. 
Lemma~\ref{lem:routine}(i) shows that the sequence $\big(\alpha_{\frac{1}{n}}\big)_{n=n_0}^\infty$ is in a compact set of equiaffine transformations, in turn having a convergent subsequence. W.l.o.g.,
$\lim_{n \to \infty} \alpha_{\frac{1}{n}}=\alpha_0 \in \AFF_2^1$. By Lemma~\ref{lem:routine}(ii),
\begin{equation}\nonumber
\lim_{n \to \infty} \alpha_{\frac{1}{n}}(D)=\alpha_0(D)
\end{equation}
in the Hausdorff metric. In particular,
\[
\textstyle
\alpha_0(D) \cap ((-1,1) \times \mathbb{R})= ((-1,1) \times \mathbb{R}) \cap \lim_{n \to \infty} \alpha_{\frac{1}{n}}(D) \cap \left(\left[-1+\frac{1}{n},1-\frac{1}{n}\right] \times \mathbb{R}\right).
\]
Now Lemma~\ref{lem:segments2}(ii) shows that $\alpha_0(D) \cap ((-1,1) \times \mathbb{R})$ is convex. Moreover, Lemma~\ref{lem:segments2}(i) yields $\alpha_0(D) \subseteq [-1,1] \times \mathbb R$. Therefore $\alpha_0(D)$ is the closure of the convex set $\alpha_0(D) \cap ((-1,1) \times \mathbb{R})$, because $\alpha_0(D)$ is a topological disc. Thus $\alpha_0(D)$ is convex, as well as $D= \alpha_0^{-1}(\alpha_0(D))$.
\end{proof}

Lemmas~\ref{lem:segments1} and \ref{lem:convex} together show that every gc-self-affine topological disc is convex.


\section{Reduction to triangles and quadrangles\label{sec:triangles_quadrangles}}

\begin{proof}[Proof of Theorem~\ref{thm:main}(i)]
Let the disc $D$ be gc-self-affine. Then it is convex as we have seen above. Moreover, it is a convex polygon by \cite{richter2012}. 

A dissection of a convex $k_0$-gon $D_0$ along some line segment splits it into a convex $k_1$-gon $D_0^*$ and a convex $k_2$-gon $D_0^{**}$ with $k_1+k_2 \in \{k_0+2,k_0+3,k_0+4\}$. In particular, $\min\{k_1,k_2\} \le \frac{k_0+4}{2}$, whence $\min\{k_1,k_2\} \le k_0$ if $k_0 \in \{3,4\}$ and $\min\{k_1,k_2\} < k_0$ if $k_0 \ge 5$.
Iteration of this argument shows that each gc-dissection of a convex $k_0$-gon $D_0$ with $k_0 \ge 5$ has a tile with less than $k_0$ vertices. Consequently, the gc-self-affine disc $D$ must be a triangle or a convex quadrangle.
\end{proof}


\section{Parametrization of quadrangles\label{sec:parametrization}}

We call a side of a convex quadrangle \emph{opening} (or \emph{closing} or \emph{constant}, respectively) if the inner angles adjacent to that side sum up to a value less than $\pi$ (or larger than $\pi$ or equal to $\pi$, respectively). Suppose that the sides $bc$ and $cd$ of the quadrangle $abcd$ are closing, and let $s$ be the intersection point of the straight lines $l(ab)$ and $l(cd)$. Then $0 < \frac{|bs|}{|as|} < \frac{|cs|}{|ds|} < 1$. We use $\alpha:=\frac{|bs|}{|as|}$ and $\beta:=\frac{|cs|}{|ds|}$ for parametrizing the quadrangle and denote its affine type by $Q(\alpha,\beta)$ (over the closing side $bc$). The same can be done over the closing side $cd$. If $t$ is the intersection of $l(bc)$ and $l(ad)$, we obtain $0 < \bar{\alpha}:=\frac{|dt|}{|at|} < \bar{\beta}:=\frac{|ct|}{|bt|} < 1$. We get the second parametrization $Q\left(\bar{\alpha},\bar{\beta}\right)$ with 
\begin{equation}\label{eq:flip}
\left(\bar{\alpha},\bar{\beta}\right)= \frac{1-\beta}{(1-\alpha)\beta}(\alpha,\beta),
\end{equation}
see Figure~\ref{fig:parameter1} for the situation $a=(0,0)$, $s=(1,0)$, $t=(0,1)$.
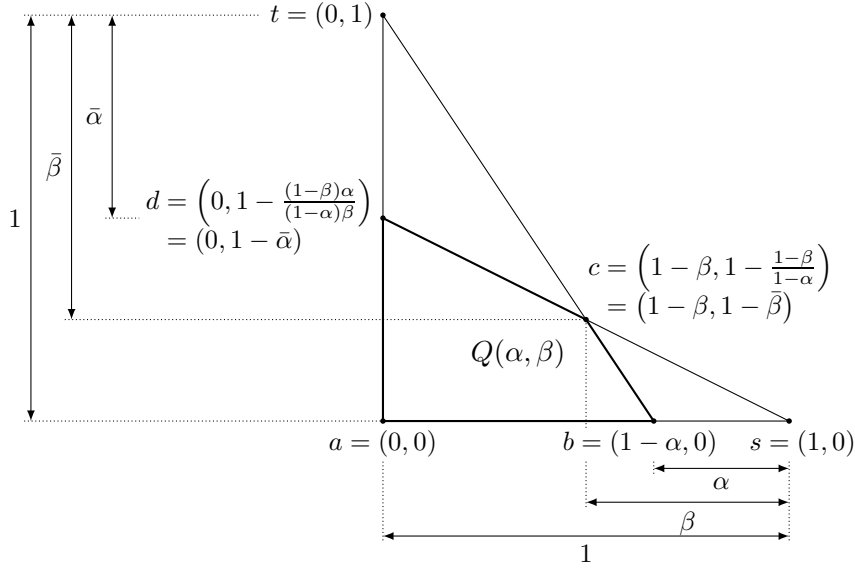
\begin{figure}
\begin{center}
\begin{tikzpicture}[scale=.9]

\fill
  (0,0) circle (.04) node[below] {$a=(0,0)$}
	(4,0) circle (.04) 
	(3.8,0) node[below] {$b=(1-\alpha,0)$}
	(6,0) circle (.04) 
	(6.2,0) node[below] {$s=(1,0)$}
  (3,1.5) circle (.04) 
	(2.7,1.3) node[above right] {$\begin{array}{r@{\;=\;}  l@{}} c&\left(1-\beta,1-\frac{  1-\beta}{1-\alpha}\right)\\
	&\left(1-\beta,1-\bar{\beta}\right)
	\end{array}$}
	(0,3) circle (.04) node[left] {$\begin{array}{r@{\;=\;}l@{}}
	d&\left(0,1-\frac{(1-\beta)\alpha}{(1-\alpha)\beta}\right)\\
	&\left(0,1-\bar{\alpha}\right)
	\end{array}$}
	(0,6) circle (.04) node[left] {$t=(0,1)$}
	(5,-.7) node[below] {$\alpha$}
	(4.5,-1.2) node[below] {$\beta$}
	(3,-1.7) node[below] {$1$}
	(-4,4.5) node[left] {$\bar{\alpha}$}
	(-4.6,3.75) node[left] {$\bar{\beta}$}
	(-5.2,3) node[left] {$1$}
	(2,1) node {\large $Q(\alpha,\beta)$}
	;

\draw[thick]
  (0,0)--(4,0)--(3,1.5)--(0,3)--cycle
  ;
  
\draw
  (0,3)--(0,6)--(3,1.5)--(6,0)--(4,0)
  ;

\draw[densely dotted]
  (0,-.6)--(0,-1.8)
  (3,1.5)--(3,-.1)
	(3,-.6)--(3,-1.3)
  (4,-.6)--(4,-0.8)
  (6,-.6)--(6,-1.8)
	(-5.3,6)--(-1.8,6)
	(-4.1,3)--(-3.6,3)
	(-4.7,1.5)--(3,1.5)
	(-5.3,0)--(0,0)
	;

\draw[latex-latex]
  (4,-.7)--(6,-.7)
  ;

\draw[latex-latex]
  (3,-1.2)--(6,-1.2)
	;
	
\draw[latex-latex]
	(0,-1.7)--(6,-1.7)
	;

\draw[latex-latex]
  (-4,6)--(-4,3)
	;
	
\draw[latex-latex]
	(-4.6,6)--(-4.6,1.5)
	;

\draw[latex-latex]
  (-5.2,6)--(-5.2,0)
  ;

\end{tikzpicture}
\end{center}
\caption{Parametrization of non-trapezoids.\label{fig:parameter1}}
\end{figure}
Both $Q(\alpha,\beta)$ and $Q\left(\bar{\alpha},\bar{\beta}\right)$ describe the same affine type of quadrangles. For further operations we keep both parametrizations and call them \emph{flips} of each other; $Q(\alpha,\beta)^F=Q\left(\bar{\alpha},\bar{\beta}\right)$ and $Q\left(\bar{\alpha},\bar{\beta}\right)^F=Q(\alpha,\beta)$. Note that $Q(\alpha,\beta)$ and $Q\left(\bar{\alpha},\bar{\beta}\right)$ give rise to the same value $0 < \frac{\alpha}{\beta}=\frac{\bar{\alpha}}{\bar{\beta}} < 1$, that we call the \emph{affine quotient} of that class of quadrangles. 

Clearly, for every choice of $0 < \alpha < \beta < 1$, $Q(\alpha,\beta)$ describes a unique affine type of convex quadrangles, and all non-trapezoidal convex quadrangles are covered this way.

Now let the quadrangle $abcd$ be a trapezoid with (unique) closing side $bc$. Again let $s$ be the common point of $l(ab)$ and $l(cd)$. Then $0 < \frac{|bs|}{|as|}=\frac{|cs|}{|ds|} < 1$, we use $\alpha:=\frac{|bs|}{|as|}$ for parametrizing the affine type of our trapezoid and write $T(\alpha)$ for the respective class. We call $1$ the \emph{affine quotient} of $T(\alpha)$ (motivated by $\frac{\alpha}{\alpha}=1$).

Clearly, all choices of $0 < \alpha < 1$ give rise to mutually different classes $T(\alpha)$, and this way all trapezoids are covered except for parallelograms. Note that the ratio of the lengths of the parallel sides of trapezoids of type $T(\alpha)$ is $\alpha$, too.

Finally, we write $P$ for the class of all parallelograms and call $1$ their \emph{affine quotient}.

In the sequel we will sometimes write $Q(\alpha,\beta)$, $T(\gamma)$ and $P$ for particular representatives of these classes of quadrangles.


\section{Parameters of compositions of two quadrangles\label{sec:quadrangles}}

When studying a gc-dissection of a disc $D$, we associate a (not necessarily unique) binary \emph{dissection tree} to its (not necessarily unique) process of dissection. The original disc $D$ is its root. In a first step $D$ is dissected into two discs, the children of $D$. Similarly, if a disc is cut further into two discs, these appear as children of that disc. Accordingly, the leafs of the tree are the tiles of the final dissection of $D$.

\begin{lem}\label{lem:tree}
All vertices of a dissection tree associated to the gc-dissection representing a gc-self-affinity of a quadrangle are quadrangles. In particular, in every step of dissection the cut is made between relative inner points of opposite sides of the respective quadrangle.
\end{lem}

\begin{proof}
Suppose in one step a quadrangle is not cut into two quadrangles. Then one of its children is a triangle. However, every further cut of a triangle produces at least one triangle. So, finally, one leaf of the tree is a triangle, a contradiction.
\end{proof}

As we know that a qc-self-affinity of a quadrangle is obtained by successively cutting parent quadrangles into two descending quadrangles, we study now how two quadrangles can be glued together along a common side to obtain a qc-dissection of a parent quadrangle.

\begin{lem}\label{lem:glueing}
All possible affine types of parent quadrangles admitting a gc-dissection into two descending quadrangles of prescribed affine types are the following.
\begin{itemize}
\item[(i)] Combinations of $Q(\alpha_1,\beta_1)$ with $Q(\alpha_2,\beta_2)$, $0 < \alpha_i < \beta_i < 1$ for $i=1,2$ (with the notation $\left(\bar{\alpha}_i,\bar{\beta}_i\right)=\frac{1-\beta_i}{(1-\alpha_i)\beta_i}(\alpha_i,\beta_i)$ as in \eqref{eq:flip}):
$$
\begin{array}{|l|l|l|}
\hline
\mbox{notation} & \mbox{(one) parametrization} & \mbox{affine}\\ 
&& \mbox{quotient}\\
\hline
\hline
&&\\[-2.3ex]
Q(\alpha_1,\beta_1) \cdot Q(\alpha_2,\beta_2) & Q(\alpha_1\alpha_2,\beta_1\beta_2) & \frac{\alpha_1}{\beta_1} \cdot \frac{\alpha_2}{\beta_2} \\
\hline
&&\\[-2.3ex]
Q(\alpha_1,\beta_1) : Q(\alpha_2,\beta_2) & \bullet \text{ if } \frac{\alpha_1}{\beta_1} < \frac{\alpha_2}{\beta_2}:\;  Q(\alpha_1\beta_2,\beta_1\alpha_2) & \frac{\alpha_1}{\beta_1}: \frac{\alpha_2}{\beta_2} \\
& \bullet \text{ if } \frac{\alpha_1}{\beta_1} = \frac{\alpha_2}{\beta_2}: \; T(\alpha_1\beta_2)=T(\beta_1\alpha_2) & 1 \\
& \bullet \text{ if } \frac{\alpha_1}{\beta_1} > \frac{\alpha_2}{\beta_2}: \; Q(\beta_1\alpha_2,\alpha_1\beta_2) & \frac{\alpha_2}{\beta_2}: \frac{\alpha_1}{\beta_1} \\
\hline
&&\\[-2.3ex]
Q(\alpha_1,\beta_1) \cdot Q(\alpha_2,\beta_2)^F & Q\left(\alpha_1\bar{\alpha}_2,\beta_1\bar{\beta}_2\right) & \frac{\alpha_1}{\beta_1} \cdot \frac{\alpha_2}{\beta_2} \\
=Q(\alpha_1,\beta_1) \cdot Q\left(\bar{\alpha}_2,\bar{\beta}_2\right) &&\\
\hline
&&\\[-2.3ex]
Q(\alpha_1,\beta_1) : Q(\alpha_2,\beta_2)^F & \bullet \text{ if } \frac{\alpha_1}{\beta_1} < \frac{\alpha_2}{\beta_2}:\;  Q\left(\alpha_1\bar{\beta}_2,\beta_1\bar{\alpha}_2\right) & \frac{\alpha_1}{\beta_1}: \frac{\alpha_2}{\beta_2} \\
=Q(\alpha_1,\beta_1) : Q\left(\bar{\alpha}_2,\bar{\beta}_2\right) 
& \bullet \text{ if } \frac{\alpha_1}{\beta_1} = \frac{\alpha_2}{\beta_2}: \; T\left(\alpha_1\bar{\beta}_2\right)=T\left(\beta_1\bar{\alpha}_2\right) & 1 \\
& \bullet \text{ if } \frac{\alpha_1}{\beta_1} > \frac{\alpha_2}{\beta_2}: \; Q\left(\beta_1\bar{\alpha}_2,\alpha_1\bar{\beta}_2\right) & \frac{\alpha_2}{\beta_2}: \frac{\alpha_1}{\beta_1} \\
\hline
&&\\[-2.3ex]
Q(\alpha_1,\beta_1)^F \cdot Q(\alpha_2,\beta_2) & Q\left(\bar{\alpha}_1\alpha_2,\bar{\beta}_1\beta_2\right) & \frac{\alpha_1}{\beta_1} \cdot \frac{\alpha_2}{\beta_2} \\
=Q\left(\bar{\alpha}_1,\bar{\beta}_1\right) \cdot Q(\alpha_2,\beta_2) &&\\
\hline
&&\\[-2.3ex]
Q(\alpha_1,\beta_1)^F : Q(\alpha_2,\beta_2) & \bullet \text{ if } \frac{\alpha_1}{\beta_1} < \frac{\alpha_2}{\beta_2}:\;  Q\left(\bar{\alpha}_1\beta_2,\bar{\beta}_1\alpha_2\right) & \frac{\alpha_1}{\beta_1}: \frac{\alpha_2}{\beta_2} \\
=Q\left(\bar{\alpha}_1,\bar{\beta}_1\right) : Q(\alpha_2,\beta_2) & \bullet \text{ if } \frac{\alpha_1}{\beta_1} = \frac{\alpha_2}{\beta_2}: \; T\left(\bar{\alpha}_1\beta_2\right)=T\left(\bar{\beta}_1\alpha_2\right) & 1 \\
& \bullet \text{ if } \frac{\alpha_1}{\beta_1} > \frac{\alpha_2}{\beta_2}: \; Q\left(\bar{\beta}_1\alpha_2,\bar{\alpha}_1\beta_2\right) & \frac{\alpha_2}{\beta_2}: \frac{\alpha_1}{\beta_1} \\
\hline
&&\\[-2.3ex]
Q(\alpha_1,\beta_1)^F \cdot Q(\alpha_2,\beta_2)^F & Q\left(\bar{\alpha}_1\bar{\alpha}_2,\bar{\beta}_1\bar{\beta}_2\right) & \frac{\alpha_1}{\beta_1} \cdot \frac{\alpha_2}{\beta_2} \\
=Q\left(\bar{\alpha}_1,\bar{\beta}_1\right) \cdot Q\left(\bar{\alpha}_2,\bar{\beta}_2\right) &&\\
\hline
&&\\[-2.3ex]
Q(\alpha_1,\beta_1)^F : Q(\alpha_2,\beta_2)^F & \bullet \text{ if } \frac{\alpha_1}{\beta_1} < \frac{\alpha_2}{\beta_2}:\;  Q\left(\bar{\alpha}_1\bar{\beta}_2,\bar{\beta}_1\bar{\alpha}_2\right) & \frac{\alpha_1}{\beta_1}: \frac{\alpha_2}{\beta_2} \\
=Q\left(\bar{\alpha}_1,\bar{\beta}_1\right) : Q\left(\bar{\alpha}_2,\bar{\beta}_2\right) 
& \bullet \text{ if } \frac{\alpha_1}{\beta_1} = \frac{\alpha_2}{\beta_2}: \; T\left(\bar{\alpha}_1\bar{\beta}_2\right)=T\left(\bar{\beta}_1\bar{\alpha}_2\right) & 1 \\
& \bullet \text{ if } \frac{\alpha_1}{\beta_1} > \frac{\alpha_2}{\beta_2}: \; Q\left(\bar{\beta}_1\bar{\alpha}_2,\bar{\alpha}_1\bar{\beta}_2\right) & \frac{\alpha_2}{\beta_2}: \frac{\alpha_1}{\beta_1} \\
\hline
\end{array}
$$

\item[(ii)] Combinations of $Q(\alpha,\beta)$, $0 < \alpha < \beta < 1$, with $T(\gamma)$, $0 < \gamma < 1$ (with the notation $\left(\bar{\alpha},\bar{\beta}\right)=\frac{1-\beta}{(1-\alpha)\beta}(\alpha,\beta)$ as in \eqref{eq:flip}):
$$
\begin{array}{|l|l|l|}
\hline
\mbox{notation} & \mbox{(one) parametrization} & \mbox{affine}\\
&& \mbox{quotient}\\
\hline
\hline
&&\\[-2.3ex]
Q(\alpha,\beta) \cdot T(\gamma) & Q(\alpha\gamma,\beta\gamma) & \frac{\alpha}{\beta} \\
\hline
&&\\[-2.3ex]
Q(\alpha,\beta)^F \cdot T(\gamma)=Q\left(\bar{\alpha},\bar{\beta}\right) \cdot T(\gamma) & Q\left(\bar{\alpha}\gamma,\bar{\beta}\gamma\right) & \frac{\alpha}{\beta} \\
\hline
\end{array}
$$

\item[(iii)] There are no possible combinations of $Q(\alpha,\beta)$, $0 < \alpha < \beta < 1$, with $P$.

\item[(iv)] Combinations of $T(\gamma_1)$ with $T(\gamma_2)$, $0 < \gamma_i < 1$ for $i=1,2$:
$$
\begin{array}{|l|l|l|}
\hline
\mbox{notation} & \mbox{possible parametrizations} & \mbox{affine}\\
&& \mbox{quotient}\\
\hline
\hline
&&\\[-2.3ex]
T(\gamma_1) \cdot T(\gamma_2) & T(\gamma_1\gamma_2) & 1\\
\hline
&&\\[-2.3ex]
T(\gamma_1)^F \cdot T(\gamma_2)^F & \bullet \text{ if } \gamma_1 \ne \gamma_2:\, \big\{T(\gamma): \min\{\gamma_1,\gamma_2\} < \gamma < 1 \big\} \cup \{P\} & 1\\
& \bullet \text{ if } \gamma_1= \gamma_2:\, \big\{T(\gamma): \min\{\gamma_1,\gamma_2\} \le \gamma < 1 \big\} \cup \{P\} & 1\\
\hline
\end{array}
$$

\item[(v)] Combinations of $T(\gamma_0)$, $0 < \gamma_0 < 1$, with $P$:
$$
\begin{array}{|l|l|l|}
\hline
\mbox{notation} & \mbox{possible parametrizations} & \mbox{affine quotient}\\
\hline
\hline
&&\\[-2.3ex]
T(\gamma_0)^F \cdot P & \big\{T(\gamma): \gamma_0 < \gamma < 1 \big\} & 1\\
\hline
\end{array}
$$

\item[(vi)] The only possible combination of $P$ with $P$ is $P$ and has the affine quotient $1$.
\end{itemize}
\end{lem}

\begin{proof}
(i). Assume that, as in Figure~\ref{fig:parameter1}, $Q(\alpha_i,\beta_i)$ is realized as the quadrangle $a_ib_ic_id_i$ with closing sides $b_ic_i$ and $c_id_i$ and with the intersection $s_i$ of $l(a_ib_i)$ with $l(c_id_i)$, such that the parameters (over $b_ic_i$) are
$$
\alpha_i=\frac{|b_is_i|}{|a_is_i|} \quad\text{ and }\quad \beta_i=\frac{|c_is_i|}{|d_is_i|}.
$$
When both $Q(\alpha_1,\beta_1)$ and $Q(\alpha_2,\beta_2)$ together form a gc-dissection of a third quadrangle, a closing side of one quadrangle must be glued together with an opening side of the other one. 

First, suppose that $b_1 c_1$ is glued with $a_2d_2$. Then necessarily $s_1=s_2=:s$. There are still two possibilities, namely either $b_1=a_2$ and $c_1=d_2$ or $b_1=d_2$ and $c_1=a_2$, see Figure~\ref{fig:glue1}. 
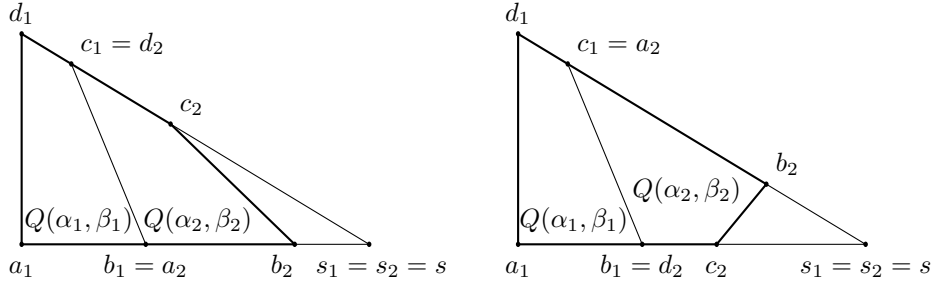
\begin{figure}
\begin{center}
\begin{tikzpicture}[xscale=.33,yscale=.4]

\fill
  (0,0) circle (.09) 
	(0,-1.4) node[above] {$a_1$}
	(5,0) circle (.09)
	(5,-1.4) node[above] {$b_1=a_2$}
	(11,0) circle (.09) 
	(10.5,-1.4) node[above] {$b_2$}
	(14,0) circle (.09) 
	(14.5,-1.4) node[above] {$s_1=s_2=s$}
	(0,7) circle (.09) node[above] {$d_1$}
	(2,6) circle (.09) node[above right] {$c_1=d_2$}
	(6,4) circle (.09) node[above right] {$c_2$}
	(2.3,.8) node {$Q(\alpha_1,\beta_1)$}
	(7.1,.8) node {$Q(\alpha_2,\beta_2)$}
  (20,0) circle (.09) 
	(20,-1.4) node[above] {$a_1$}
	(25,0) circle (.09)
	(25,-1.4) node[above] {$b_1=d_2$}
	(28,0) circle (.09) 
	(28,-1.4) node[above] {$c_2$}
	(34,0) circle (.09) 
	(34,-1.4) node[above] {$s_1=s_2=s$}
	(20,7) circle (.09) node[above] {$d_1$}
	(22,6) circle (.09) node[above right] {$c_1=a_2$}
	(30,2) circle (.09) node[above right] {$b_2$}
	(22.25,0.8) node {$Q(\alpha_1,\beta_1)$}
	(26.8,1.8) node {$Q(\alpha_2,\beta_2)$}
	;

\draw[thick]
  (0,0)--(11,0)--(6,4)--(0,7)--cycle
  (20,0)--(28,0)--(30,2)--(20,7)--cycle
  ;
  
\draw
  (11,0)--(14,0)--(6,4)
	(5,0)--(2,6)
  (28,0)--(34,0)--(30,2)
	(25,0)--(22,6)
  ;

\end{tikzpicture}
\end{center}
\caption{$Q(\alpha_1,\beta_1) \cdot Q(\alpha_2,\beta_2)$ and $Q(\alpha_1,\beta_1) : Q(\alpha_2,\beta_2)$.\label{fig:glue1}}
\end{figure}
In the former situation we obtain a quadrangle $Q(\alpha_1,\beta_1) \cdot Q(\alpha_2,\beta_2)$ with vertices $a_1b_2c_2d_1$ and parameters
$$
0 < \alpha=\frac{|b_2s|}{|a_1s|}= \frac{|b_1s_1|}{|a_1s_1|} \cdot \frac{|b_2s_2|}{|a_2s_2|}= \alpha_1\alpha_2 < \beta=\frac{|c_2s|}{|d_1s|}= \frac{|c_1s_1|}{|d_1s_1|} \cdot \frac{|c_2s_2|}{|d_2s_2|}= \beta_1\beta_2 < 1,
$$
which yields the first line in the table of (i). In the latter situation we obtain a quadrangle $Q(\alpha_1,\beta_1) : Q(\alpha_2,\beta_2)$ with vertices $a_1c_2b_2d_1$. The parameters are computed by
$$
0< \frac{|c_2s|}{|a_1s|}= \frac{|b_1s_1|}{|a_1s_1|} \cdot \frac{|c_2s_2|}{|d_2s_2|}= \alpha_1\beta_2 < 1 \;\text{ and }\;
0 < \frac{|b_2s|}{|d_1s|}= \frac{|c_1s_1|}{|d_1s_1|} \cdot \frac{|b_2s_2|}{|a_2s_2|}= \beta_1\alpha_2 < 1.
$$
Depending on whether $\alpha_1\beta_2 \begin{array}{c} < \\[-1.6ex] = \\[-1.6ex] > \end{array}
\alpha_2\beta_1$ we obtain the parametrization and the affine quotient as in the second line in the table of (i).
 
Note that the resulting affine types are exactly the same if $a_1d_1$ is glued together with $b_2c_2$. So we obtain 
\begin{equation}\label{eq:commutative}
\begin{array}{lcr}
Q(\alpha_1,\beta_1)\cdot Q(\alpha_2,\beta_2)&=&Q(\alpha_2,\beta_2)\cdot Q(\alpha_1,\beta_1),\\ 
Q(\alpha_1,\beta_1): Q(\alpha_2,\beta_2)&=&Q(\alpha_2,\beta_2): Q(\alpha_1,\beta_1).
\end{array}
\end{equation}

The remainder of (i) is obtained by glueing $b_1c_1$ with $a_2b_2$ (or, equivalently, $a_1b_1$ with $b_2c_2$), $c_1 d_1$ with $a_2 d_2$ (or $a_1d_1$ with $c_2d_2$), and $c_1d_1$ with $a_2b_2$ (or $a_1b_1$ with $c_2d_2$). The distinction of cases and the resulting affine quotient can be formulated in terms of $\frac{\alpha_i}{\beta_i}$, since $\frac{\bar{\beta}_i}{\bar{\alpha}_i}=\frac{\alpha_i}{\beta_i}$.

(ii). Here the arguments are as in (i). The situation simplifies, because $T(\gamma)$ has only one parameter and only one opening and one closing side.

(iii). $Q(\alpha,\beta)$ and $P$ cannot be combined, since opposite sides of $P$ are parallel, whereas opposite sides of $Q(\alpha,\beta)$ are not.

(iv). The first line of the table of (iv), i.e. the quadrangle $T(\gamma_1) \cdot T(\gamma_2)$, is obtained by glueing the closing side of $T(\gamma_1)$ together with the opening side of $T(\gamma_2)$ or vice versa.

Alternatively, $T(\gamma_1)$ and $T(\gamma_2)$ can be glued along constant sides, see Figure~\ref{fig:glue2}. This type of glueing is called $T(\gamma_1)^F \cdot T(\gamma_2)^F$.
\begin{figure}
\begin{center}
\begin{tikzpicture}[scale=.5]

\draw[thick]
  (0,0)--(9,0)--(6,3)--(0,3)--cycle
  (13,0)--(20,0)--(21,3)--(13,3)--cycle
  ;
  
\draw
  (4,0)--(3,3)
	(17,0)--(16,3)
  (1.75,1.5) node {$T(\gamma_1)$}
  (5.5,1.5) node {$T(\gamma_2)$}
  (14.75,1.5) node {$T(\gamma_1)$}
  (18.5,1.5) node {$T(\gamma_2)$}
  (2,0) node[below] {$1$}
  (1.5,3) node[above] {$\gamma_1$}
  (6.5,0) node[below] {$\lambda$}
  (4.5,3) node[above] {$\lambda\gamma_2$}
  (15,0) node[below] {$1$}
  (14.5,3) node[above] {$\gamma_1$}
  (18.5,0) node[below] {$\lambda\gamma_2$}
  (18.5,3) node[above] {$\lambda$}
	;

\end{tikzpicture}
\end{center}
\caption{Combining $T(\gamma_1)$ with $T(\gamma_2)$ along constant sides.\label{fig:glue2}}
\end{figure}
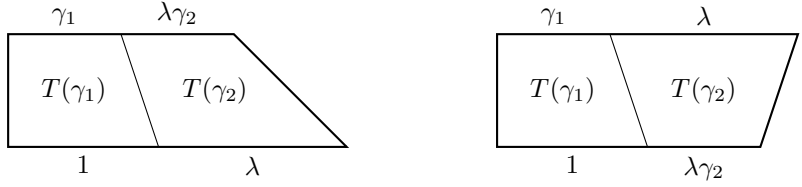
We can assume that, w.l.o.g., $\gamma_1 \le \gamma_2$ and the parallel sides of $T(\gamma_1)$ have lengths $1$ and $\gamma_1$. Then the lengths of the parallel sides of the combined trapezoids are either $\gamma_1 + \lambda \gamma_2$ and $1+\lambda$ or $\gamma_1+\lambda$ and $1+\lambda\gamma_2$ with arbitrary $\lambda > 0$. Their respective ratios are
$$
\frac{\gamma_1+\lambda\gamma_2}{1+\lambda}=\frac{1}{1+\lambda} \gamma_1+ \frac{\lambda}{1+\lambda} \gamma_2,
$$
which is $\gamma_1$ if $\gamma_1=\gamma_2$ or ranges in the interval $(\gamma_1,\gamma_2)$ if $\gamma_1 < \gamma_2$,
and
$$
\frac{\gamma_1+\lambda}{1+\lambda\gamma_2}=\frac{1}{1+\lambda\gamma_2} \gamma_1+ \frac{\lambda\gamma_2}{1+\lambda\gamma_2} \frac{1}{\gamma_2},
$$
which ranges in $\left(\gamma_1,\frac{1}{\gamma_2}\right) \supseteq (\gamma_1,1]$. Consequently, the affine types obtained by $T(\gamma_1)^F \cdot T(\gamma_2)^F$ are $T(\gamma)$, $\gamma_1 \le \gamma< 1$, and $P$ if $\gamma_1=\gamma_2$, and $T(\gamma)$, $\gamma_1 < \gamma< 1$, and $P$ if $\gamma_1<\gamma_2$.

(v). This is obtained as (iv). Only glueing along constant sides is possible.

(vi). This is trivial.
\end{proof}

\begin{rem}
\begin{itemize}
\item[(i)] 
The multiplicative operational notation comes with a multiplication of the parameters as well as of the affine quotients. The divisional notation corresponds to a division of the affine quotients and to an exchange of the parameters of one quadrangle.

\item[(ii)]
We do not use a flip in our operational notation if the respective glueing is made along the closing side over which the parametrization is made (or along its opposite opening side). We use a flip if the glueing is made along the neighbouring closing (or opening) side in case of a non-trapezoid or along a constant side in case of a trapezoid.

\item[(iii)]
The operational notations $T(\gamma_1)^F \cdot T(\gamma_2)^F$ and $T(\gamma_0)^F \cdot P$ stand for infinitely many possible results, whereas the other notations produce a unique outcome. The notations are commutative (see \eqref{eq:commutative} in the proof of (i)).

\item[(iv)]
Lemma~\ref{lem:glueing} allows to classify all $2$-gc-self-affine quadrangles (and all $2$-self-affine convex quadrangles, since this is equivalent). Indeed, every trapezoid is $2$-gc-self-affine, as follows from Lemma~\ref{lem:glueing}(iv) and (vi) (but in fact is trivial, see Figure~\ref{fig:1}). On the other hand, non-trapezoids are not $2$-gc-self-affine, because Lemma~\ref{lem:glueing}(i) shows that a combination of two copies of $Q(\alpha,\beta)$, $0 < \alpha < \beta < 1$, has an affine quotient $\big(\frac{\alpha}{\beta}\big)^2$ or $1$, which is different from the quotient $\frac{\alpha}{\beta}$ of $Q(\alpha,\beta)$.
\end{itemize}
\end{rem}


\section{$n$-gc-self-affine quadrangles with even $n$ are trapezoids\label{sec:trapezoids}}

\begin{lem}\label{lem:parity}
Suppose that a quadrangle $Q$ has a gc-dissection into $n$ copies of $Q(\alpha,\beta)$, $0 < \alpha < \beta < 1$. Then the affine quotient of $Q$ is of the form $\big(\frac{\alpha}{\beta}\big)^k$ with $k \in \{0,1,\ldots,n\}$, where $k$ is even if $n$ is even and $k$ is odd if $n$ is odd.
\end{lem}

\begin{proof}
This is obtained by mathematical induction over the number $n$ of tiles (or of leafs in the dissection tree). The claim is trivial for the trivial dissection ($n=1$). In case of a finer dissection $Q$ is split into two quadrangles $Q^*$ and $Q^{**}$ who themselves are dissected into $n_1$ and $n_2$ tiles with $n_1+n_2=n$ and $0 < n_1,n_2 < n$. By the induction hypothesis, their affine quotients are $\big(\frac{\alpha}{\beta}\big)^{k_1}$ and $\big(\frac{\alpha}{\beta}\big)^{k_2}$ with $k_i \in \{0,1,\ldots,n_i\}$ and $k_i \equiv n_i \mod 2$, $i=1,2$. Now Lemma~\ref{lem:glueing} shows that $Q=Q^* \cup Q^{**}$ has an affine quotient $\big(\frac{\alpha}{\beta}\big)^{k}$ with $k\in \{k_1+k_2,|k_1-k_2|\}$. Accordingly, $k \in \{0,1,\ldots,k_1+k_2\}\subseteq\{0,1,\ldots,n_1+n_2\}$ and $k\equiv k_1+k_2 \equiv n_1+n_2=n \mod 2$.
\end{proof}

\begin{proof}[Proof of Theorem~\ref{thm:main}(iii)]
It is clear that every trapezoid is $n$-gc-self-affine by repeated application of Lemma~\ref{lem:glueing}(iv) and (vi), see also Figure~\ref{fig:1}. On the other hand, a non-trapezoid $Q(\alpha,\beta)$, $0 < \alpha < \beta < 1$, cannot be gc-dissected into an even number $n$ of copies of $Q(\alpha,\beta)$, because then Lemma~\ref{lem:parity} implied that $Q(\alpha,\beta)$ had an affine quotient $\big(\frac{\alpha}{\beta}\big)^k$ with even $k$, which is different from the actual quotient $\big(\frac{\alpha}{\beta}\big)^1$.
\end{proof}


\section{$n$-gc-self-affine quadrangles for odd $n \ge 5$\label{sec:n>=5}}

\begin{lem}\label{lem:trapezoids}
Let $0 < \alpha < \beta < 1$.
\begin{itemize}
\item[(i)]
$T(\alpha\beta)$ has a gc-dissection into two copies of $Q(\alpha,\beta)$.

\item[(ii)]
For every real number $\gamma$ with $\alpha\beta\min\left\{\frac{1-\beta}{(1-\alpha)\beta},1\right\}\le \gamma < 1$ and every even integer $k \ge 4$, $T(\gamma)$ has a gc-dissection into $k$ copies of $Q(\alpha,\beta)$.
\end{itemize}
\end{lem} 

\begin{proof}
(i). This comes from $Q(\alpha,\beta):Q(\alpha,\beta)$ in Lemma~\ref{lem:glueing}(i).

(ii). First suppose that $\frac{1-\beta}{(1-\alpha)\beta} \ge 1$. By (i), we can produce $\frac{k}{2}$ copies of $T(\alpha\beta)$, each composed of two copies of $Q(\alpha,\beta)$. Repeated use of the first part of Lemma~\ref{lem:glueing}(iv) shows that $T\left((\alpha\beta)^{\frac{k}{2}-1}\right)$ can be obtained from $\frac{k}{2}-1$ copies of $T(\alpha\beta)$, that is, from $k-2$ copies of $Q(\alpha,\beta)$. Finally, we use the second part of Lemma~\ref{lem:glueing}(iv) for writing $T(\gamma)$ with $\alpha\beta \le \gamma < 1$ as $T\left((\alpha\beta)^{\frac{k}{2}-1}\right)^F\cdot T(\alpha\beta)^F$, this way obtaining a gc-dissection into $(k-2)+2=k$ copies of $Q(\alpha,\beta)$.

If $\frac{1-\beta}{(1-\alpha)\beta} < 1$, we recall that $Q(\alpha,\beta)$ represents the same quadrangles as 
$$
Q(\alpha,\beta)^F=Q\left(\frac{1-\beta}{(1-\alpha)\beta}(\alpha,\beta)\right)=Q\left(\bar{\alpha},\bar{\beta}\right),
$$ 
where $0 < \bar{\alpha} < \bar{\beta} < 1$ and $\frac{1-\bar{\beta}}{(1-\bar{\alpha})\bar{\beta}}=\left(\frac{1-\beta}{(1-\alpha)\beta}\right)^{-1} < 1$.
As above, we see that $T(\gamma)$ has a gc-dissection into $k$ copies of $Q\left(\bar{\alpha},\bar{\beta}\right)$ whenever $\gamma$ is in the interval 
$$
\left[\bar{\alpha}\bar{\beta},1\right)= \left[\left(\frac{1-\beta}{(1-\alpha)\beta}\right)^2\alpha\beta,1\right) \supseteq \left[\frac{1-\beta}{(1-\alpha)\beta}\alpha\beta,1\right).
$$
\end{proof}

In the sequel affine images of (convex) kites turn out to be crucial. We call them \emph{affine kites} for short. Of course, a trapezoid is an affine kite if and only if it is a parallelogram.

\begin{lem}\label{lem:kites}
Let $0 < \alpha < \beta < 1$. The following are equivalent.
\begin{itemize}
\item[(i)]
$Q(\alpha,\beta)$ is an affine kite.
\item[(ii)]
$\frac{1-\beta}{(1-\alpha)\beta}=1$ (or, equivalently, $\alpha=2-\frac{1}{\beta}$ or $\beta=\frac{1}{2-\alpha}$).
\item[(iii)]
The parametrization $Q(\alpha,\beta)$ is invariant under flipping: $Q(\alpha,\beta)^F=Q(\alpha,\beta)$.
\end{itemize}
In particular, every affine kite has a unique parametrization.
\end{lem}

\begin{proof}
We refer to the representation of $Q(\alpha,\beta)$ given in Figure~\ref{fig:parameter1}. Then $Q(\alpha,\beta)$ is an affine kite if and only if the diagonal $ac$ meets the midpoint of $bd$. Equivalently, the linear map given by $a \mapsto a$, $s \mapsto t$ and $t \mapsto s$ maps $c=\left(1-\beta,1-\bar{\beta}\right)$ onto itself. Since that map is nothing but the reflection exchanging the coordinates, $Q(\alpha,\beta)$ is an affine kite if and only if $\bar{\beta}=\beta$. Taking into account that $\left(\bar{\alpha},\bar{\beta}\right)=\frac{1-\beta}{(1-\alpha)\beta}(\alpha,\beta)$, we see the equivalence of (i), (ii) and (iii).
\end{proof}

\begin{lem}\label{lem:odd_positive}
\begin{itemize}
\item[(i)]
If $Q(\alpha,\beta)$, $0 < \alpha < \beta < 1$, is no affine kite, then $Q(\alpha,\beta)$ is $n$-gc-self-affine for every odd $n \ge 5$.

\item[(ii)]
Every affine kite is $n$-gc-self-affine for every odd $n \ge 7$.
\end{itemize}
\end{lem}

\begin{proof}
(i). By Lemma~\ref{lem:kites}, $\frac{1-\beta}{(1-\alpha)\beta} \ne 1$. As in the proof of Lemma~\ref{lem:trapezoids}, we can assume that $\frac{1-\beta}{(1-\alpha)\beta} < 1$, since otherwise we consider $Q(\alpha,\beta)^F$ instead of $Q(\alpha,\beta)$. By Lemma~\ref{lem:trapezoids}(ii), $T\left(\frac{1-\beta}{(1-\alpha)\beta}\right)$ has a gc-dissection into $n-1$ copies of $Q(\alpha,\beta)$. Now Lemma~\ref{lem:glueing}(ii) shows that
$$
Q(\alpha,\beta)^F=Q\left(\frac{1-\beta}{(1-\alpha)\beta}(\alpha,\beta)\right)=Q(\alpha,\beta) \cdot T\left(\frac{1-\beta}{(1-\alpha)\beta}\right)
$$
has a gc-dissection into $n$ copies of $Q(\alpha,\beta)$, which yields (i).

(ii). The claim is trivial for trapezoidal affine kites (i.e., for parallelograms). It remains to consider a kite $Q(\alpha,\beta)$, $0 < \alpha < \beta < 1$. By Lemma~\ref{lem:kites}, $\beta=\frac{1}{2-\alpha}$. An easy calculation shows that the number $\gamma=\frac{(1-\alpha^2\beta)\beta}{1-\alpha\beta^2}$ satisfies $\alpha\beta \le \gamma < 1$. Then Lemma~\ref{lem:trapezoids} shows that $T(\alpha\beta)$ has a gc-dissection into two copies of $Q(\alpha,\beta)$ and $T(\gamma)$ has a gc-dissection into $n-3$ copies of $Q(\alpha,\beta)$. Another simple calculation gives
$$
Q(\alpha,\beta)=(Q(\alpha,\beta) \cdot T(\alpha\beta))^F \cdot T(\gamma),
$$
which in turn yields that $Q(\alpha,\beta)$ has a gc-dissection into $(1+2)+(n-3)=n$ copies of $Q(\alpha,\beta)$.
\end{proof}

\begin{lem}\label{lem:kites5}
If $Q(\alpha,\beta)$, $0 < \alpha < \beta < 1$, is an affine kite, then $Q(\alpha,\beta)$ is not $5$-gc-self-affine.
\end{lem}

\begin{proof}
Suppose that some affine kite $Q\left(\alpha,\frac{1}{2-\alpha}\right)$, $0 < \alpha < 1$ (cf. Lemma~\ref{lem:kites}), is $5$-gc-self-affine.

\emph{Step 1. Reduction to five extended dissection trees. }
The dissection tree has five leafs, called $L$, and a root $R$. 
$L$ and $R$ have the unique parametrization $Q\left(\alpha,\frac{1}{2-\alpha}\right)$, see Lemma~\ref{lem:kites}(iii).
We extend the tree as follows. Given a vertex $Q$ that is not a leaf, a parametrization of $Q$ is obtained from the parametrizations of its two children $C_1$ and $C_2$ by an operation $\cdot$ or $:$, possibly preceded by a flip on $C_1$ and/or $C_2$, see Lemma~\ref{lem:glueing}. We mark the operation $\cdot$ or $:$ at the vertex $Q$ and flips, in case they exist, at the corresponding edges between $Q$ and its children, see Figure~\ref{fig:kites}.
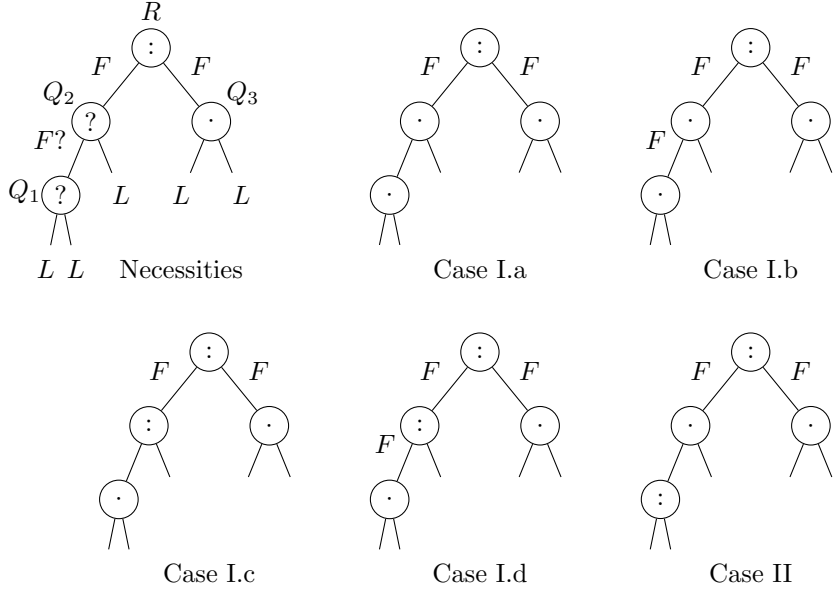
\begin{figure}
\begin{center}
\begin{tikzpicture}[scale=.45]

\draw (0,.51) node{\tikz[xscale=.4,yscale=.49]{

  \node[shape=circle,draw=black] (R) at (0,0) {$:$};
  \node[shape=circle,draw=black] (Q2) at (-2,-2) {\!\!$?$\!\!};
  \node[shape=circle,draw=black] (Q3) at (2,-2) {$\cdot$};
  \node[shape=circle,draw=black] (Q1) at (-3,-4) {\!\!$?$\!\!};
  \node[shape=circle] (L1) at (-3.5,-6) {$\quad$};
  \node[shape=circle] (L2) at (-2.5,-6) {$\quad$};
  \node[shape=circle] (L3) at (-1,-4) {$\quad$};
  \node[shape=circle] (L4) at (1,-4) {$\quad$};
  \node[shape=circle] (L5) at (3,-4) {$\quad$};

\path [-] (R) edge node[above left] {$F$} (Q2);
\path [-] (R) edge node[above right] {$F$} (Q3);
\path [-] (Q2) edge node[above left] {$F?$} (Q1);
\path [-] (Q1) edge (L1);
\path [-] (Q1) edge (L2);
\path [-] (Q2) edge (L3);
\path [-] (Q3) edge (L4);
\path [-] (Q3) edge (L5);

\draw
  (0,.5) node[above] {$R$}
  (-3.35,-3.9) node[left] {$Q_1$}
  (-2.2,-1.8) node[above left] {$Q_2$}
  (2.2,-1.8) node[above right] {$Q_3$}
  (L1) node {$L$}
  (L2) node {$L$}
  (L3) node {$L$}
  (L4) node {$L$}
  (L5) node {$L$}
  (1,-6) node {Necessities}
	;

}}
;

\draw (10,0) node{\tikz[xscale=.4,yscale=.49]{

  \node[shape=circle,draw=black] (R) at (0,0) {$:$};
  \node[shape=circle,draw=black] (Q2) at (-2,-2) {$\cdot$};
  \node[shape=circle,draw=black] (Q3) at (2,-2) {$\cdot$};
  \node[shape=circle,draw=black] (Q1) at (-3,-4) {$\cdot$};
  \node[shape=circle] (L1) at (-3.5,-6) {$\quad$};
  \node[shape=circle] (L2) at (-2.5,-6) {$\quad$};
  \node[shape=circle] (L3) at (-1,-4) {$\quad$};
  \node[shape=circle] (L4) at (1,-4) {$\quad$};
  \node[shape=circle] (L5) at (3,-4) {$\quad$};

\path [-] (R) edge node[above left] {$F$} (Q2);
\path [-] (R) edge node[above right] {$F$} (Q3);
\path [-] (Q2) edge node[above left] {} (Q1);
\path [-] (Q1) edge (L1);
\path [-] (Q1) edge (L2);
\path [-] (Q2) edge (L3);
\path [-] (Q3) edge (L4);
\path [-] (Q3) edge (L5);

\draw
  (0,-6) node {Case I.a}
	;

}}
;

\draw (18,0) node{\tikz[xscale=.4,yscale=.49]{

  \node[shape=circle,draw=black] (R) at (0,0) {$:$};
  \node[shape=circle,draw=black] (Q2) at (-2,-2) {$\cdot$};
  \node[shape=circle,draw=black] (Q3) at (2,-2) {$\cdot$};
  \node[shape=circle,draw=black] (Q1) at (-3,-4) {$\cdot$};
  \node[shape=circle] (L1) at (-3.5,-6) {$\quad$};
  \node[shape=circle] (L2) at (-2.5,-6) {$\quad$};
  \node[shape=circle] (L3) at (-1,-4) {$\quad$};
  \node[shape=circle] (L4) at (1,-4) {$\quad$};
  \node[shape=circle] (L5) at (3,-4) {$\quad$};

\path [-] (R) edge node[above left] {$F$} (Q2);
\path [-] (R) edge node[above right] {$F$} (Q3);
\path [-] (Q2) edge node[above left] {$F$} (Q1);
\path [-] (Q1) edge (L1);
\path [-] (Q1) edge (L2);
\path [-] (Q2) edge (L3);
\path [-] (Q3) edge (L4);
\path [-] (Q3) edge (L5);

\draw
  (0,-6) node {Case I.b}
	;

}}
;

\draw (2,-9) node{\tikz[xscale=.4,yscale=.49]{

  \node[shape=circle,draw=black] (R) at (0,0) {$:$};
  \node[shape=circle,draw=black] (Q2) at (-2,-2) {$:$};
  \node[shape=circle,draw=black] (Q3) at (2,-2) {$\cdot$};
  \node[shape=circle,draw=black] (Q1) at (-3,-4) {$\cdot$};
  \node[shape=circle] (L1) at (-3.5,-6) {$\quad$};
  \node[shape=circle] (L2) at (-2.5,-6) {$\quad$};
  \node[shape=circle] (L3) at (-1,-4) {$\quad$};
  \node[shape=circle] (L4) at (1,-4) {$\quad$};
  \node[shape=circle] (L5) at (3,-4) {$\quad$};

\path [-] (R) edge node[above left] {$F$} (Q2);
\path [-] (R) edge node[above right] {$F$} (Q3);
\path [-] (Q2) edge node[above left] {} (Q1);
\path [-] (Q1) edge (L1);
\path [-] (Q1) edge (L2);
\path [-] (Q2) edge (L3);
\path [-] (Q3) edge (L4);
\path [-] (Q3) edge (L5);

\draw
  (0,-6) node {Case I.c}
	;

}}
;

\draw (10,-9) node{\tikz[xscale=.4,yscale=.49]{

  \node[shape=circle,draw=black] (R) at (0,0) {$:$};
  \node[shape=circle,draw=black] (Q2) at (-2,-2) {$:$};
  \node[shape=circle,draw=black] (Q3) at (2,-2) {$\cdot$};
  \node[shape=circle,draw=black] (Q1) at (-3,-4) {$\cdot$};
  \node[shape=circle] (L1) at (-3.5,-6) {$\quad$};
  \node[shape=circle] (L2) at (-2.5,-6) {$\quad$};
  \node[shape=circle] (L3) at (-1,-4) {$\quad$};
  \node[shape=circle] (L4) at (1,-4) {$\quad$};
  \node[shape=circle] (L5) at (3,-4) {$\quad$};

\path [-] (R) edge node[above left] {$F$} (Q2);
\path [-] (R) edge node[above right] {$F$} (Q3);
\path [-] (Q2) edge node[above left] {$F$} (Q1);
\path [-] (Q1) edge (L1);
\path [-] (Q1) edge (L2);
\path [-] (Q2) edge (L3);
\path [-] (Q3) edge (L4);
\path [-] (Q3) edge (L5);

\draw
  (0,-6) node {Case I.d}
	;

}}
;

\draw (18,-9) node{\tikz[xscale=.4,yscale=.49]{

  \node[shape=circle,draw=black] (R) at (0,0) {$:$};
  \node[shape=circle,draw=black] (Q2) at (-2,-2) {$\cdot$};
  \node[shape=circle,draw=black] (Q3) at (2,-2) {$\cdot$};
  \node[shape=circle,draw=black] (Q1) at (-3,-4) {$:$};
  \node[shape=circle] (L1) at (-3.5,-6) {$\quad$};
  \node[shape=circle] (L2) at (-2.5,-6) {$\quad$};
  \node[shape=circle] (L3) at (-1,-4) {$\quad$};
  \node[shape=circle] (L4) at (1,-4) {$\quad$};
  \node[shape=circle] (L5) at (3,-4) {$\quad$};

\path [-] (R) edge node[above left] {$F$} (Q2);
\path [-] (R) edge node[above right] {$F$} (Q3);
\path [-] (Q2) edge node[above left] {} (Q1);
\path [-] (Q1) edge (L1);
\path [-] (Q1) edge (L2);
\path [-] (Q2) edge (L3);
\path [-] (Q3) edge (L4);
\path [-] (Q3) edge (L5);

\draw
  (0,-6) node {Case II}
	;

}}
;

\end{tikzpicture}
\end{center}
\caption{Particular extended dissection trees for the proof of Lemma~\ref{lem:kites5}.\label{fig:kites}}
\end{figure}

We collect necessary properties of the extended dissection tree.
\begin{itemize}
\item[(A)] 
W.l.o.g., no edge emanating from a leaf has a flip.

\item[(B)]
On the path between any leaf $L$ and the root $R$ there is at least one flip.

\end{itemize}

For (A), note that leafs represent kites, that are invariant under flips. 

For (B), suppose that there is no flip on the path from some $L$ to $R$, both with unique parametrization $Q\left(\alpha, \frac{1}{2-\alpha}\right)$. Then, by Lemma~\ref{lem:glueing}, in each operation $\cdot$ or $:$ on the path from $L$ to $R$, the smaller parameter of the respective quadrangle decreases strictly, since it is multiplied by some positive number less than one. Consequently, the smaller parameter of $R$ is smaller than that of $L$, contradicting that both have the same unique parametrization $Q\left(\alpha, \frac{1}{2-\alpha}\right)$.


It follows from (A) and (B) that no child of $R$ is a leaf. Hence there are two sub-trees below $R$, one with $3$ and one with $2$ leafs. Using the commutativity of $\cdot$ and $:$ (cf. \eqref{eq:commutative}), the root $R$, leafs $L$, additional vertices $Q_1$, $Q_2$, $Q_3$ and the edges of the tree are as in the upper left illustration in Figure~\ref{fig:kites}. Moreover, by (A) and (B) we have no flips on the edges emanating from leafs and we have flips on both edges emanating from $R$.

\begin{itemize}
\item[(C)]
The extended dissection tree has the structure displayed under Necessities in Figure~\ref{fig:kites}. In particular, we have the operations $:$ at $R$ and $\cdot$ at $Q_3$, and only for the edge between $Q_1$ and $Q_2$ it is not yet determined if there is a flip.
\end{itemize}

The situation for flips is already shown. 

For the operations at $R$ and $Q_3$, suppose first that we have $\cdot$ at $R$. By Lemma~\ref{lem:parity}, the affine quotient of $Q_2$ and $Q_3$ are $[\alpha(2-\alpha)]^{k_2}$ with $k_2 \in \{1,3\}$ and $[\alpha(2-\alpha)]^{k_3}$ with $k_3 \in \{0,2\}$, respectively. Since $R=Q_2^F \cdot Q_3^F$, we have the quotient $\alpha(2-\alpha)=[\alpha(2-\alpha)]^{k_2} \cdot [\alpha(2-\alpha)]^{k_3}$ by Lemma~\ref{lem:glueing}, which yields $k_2+k_3=1$ and in turn $k_2=1$ and $k_3=0$. By $k_3=0$, we have $:$ at $Q_3$ and $Q_3=T$ is a trapezoid. But then the equation $R=Q_2^F \cdot T^F$ implies that $Q_2^F$ must be a flipped trapezoid or a parallelogram, since these are the only partners for multiplication with flipped trapezoids by Lemma~\ref{lem:glueing}. However $Q_2^F$ cannot be of that type, because its affine quotient is $[\alpha(2-\alpha)]^1\ne 1$. This contradiction shows that the operation at $R$ is $:$. 

Finally, to see that we have $\cdot$ at $Q_3$, assume to the contrary that there is $:$ at $Q_3$. But then $Q_3$ is a trapezoid $T$, and an operation $R=Q_2^F:T^F$ does not exist. So (C) is verified.

Now we discuss cases depending on the operations at $Q_1$ and $Q_2$ and on the existence of a flip between $Q_1$ and $Q_2$.

\emph{Case I, there is $\cdot$ at $Q_1$. } Then we get Case I.a, \ldots, Case I.d from Figure~\ref{fig:kites}.

\emph{Case II, there is $:$ at $Q_1$. } Then $Q_1$ is a trapezoid $T$, which does not allow $:$ at $Q_2$. So there is $\cdot$ at $Q_2$. Moreover, there is no flip between $Q_1=T$ and $Q_2$, because this would give the impossibility $Q_2=T^F\cdot L$, since $L$ is neither a parallelogram nor a flipped trapezoid. We obtain the last situation illustrated in Figure~\ref{fig:kites}.

\emph{Step 2. Discussion of the five remaining cases. } For every case we compute the parametrization of $R$ along the extended dissection tree, based on $L=Q\left(\alpha,\frac{1}{2-\alpha}\right)$, and we analyse the condition $R=Q\left(\alpha,\frac{1}{2-\alpha}\right)$. The computations can be made by hand or by some computer algebra system.

\emph{Case I.a. } Here $R=((L\cdot L)\cdot L)^F:(L\cdot L)^F$ gives
$$
R=Q\left(\frac{(\alpha^2-5\alpha+7)(3-\alpha)\alpha^2}{(\alpha^2+\alpha+1)(\alpha+1)(2-\alpha)^2}\left(\alpha,\frac{1}{2-\alpha}\right)\right).
$$
The condition $R=Q\left(\alpha,\frac{1}{2-\alpha}\right)$ amounts to
$1=\frac{(\alpha^2-5\alpha+7)(3-\alpha)\alpha^2}{(\alpha^2+\alpha+1)(\alpha+1)(2-\alpha)^2}$. Subtraction of $1$, multiplication by the denominator and division by the non-zero term $2(1-\alpha)$ yields the contradiction
$$
0=\alpha^4-4\alpha^3+6\alpha^2-4\alpha-2=(1-\alpha)^4-3 < 1-3=-2,
$$
because $0 < \alpha < 1$.
Therefore Case I.a does not give a $5$-gc-self-affinity of an affine kite.

\emph{Case I.b. } Here $R=((L\cdot L)^F\cdot L)^F:(L\cdot L)^F$ gives
$$
R=Q\left(\frac{(-\alpha^3+4\alpha^2-2\alpha-5)(3-\alpha)\alpha^2}{(\alpha^3-2\alpha^2-2\alpha-1)(\alpha+1)(2-\alpha)^2}\left(\alpha,\frac{1}{2-\alpha}\right)\right).
$$
The condition $R=Q\left(\alpha,\frac{1}{2-\alpha}\right)$ reads as
$1=\frac{(-\alpha^3+4\alpha^2-2\alpha-5)(3-\alpha)\alpha^2}{(\alpha^3-2\alpha^2-2\alpha-1)(\alpha+1)(2-\alpha)^2}$. Subtracting $1$, multiplying by the denominator and dividing by $2(1-\alpha)$ gives the contradiction
$$
0=\alpha^4-4\alpha^3+\alpha^2+6\alpha+2=\alpha^2(2-\alpha)^2-3(1-\alpha)^2+5 > 0-3+5=2
$$
for $0 < \alpha < 1$.

\emph{Case I.c. } Here $R=((L\cdot L): L)^F:(L\cdot L)^F$ gives
$$
R=Q\left(\frac{(4-\alpha)(3-\alpha)\alpha}{(\alpha+2)(\alpha+1)(2-\alpha)}\left(\alpha,\frac{1}{2-\alpha}\right)\right).
$$
The condition $R=Q\left(\alpha,\frac{1}{2-\alpha}\right)$ is now
$1=\frac{(4-\alpha)(3-\alpha)\alpha}{(\alpha+2)(\alpha+1)(2-\alpha)}$. Subtracting $1$, multiplying by the denominator and dividing by $2(1-\alpha)$ gives the contradiction
$$
0=-\alpha^2+2\alpha-2=-(1-\alpha)^2-1 < -1
$$
for $0 < \alpha < 1$.

\emph{Case I.d. } Here $R=((L\cdot L)^F: L)^F:(L\cdot L)^F$ gives
$$
R=Q\left(\frac{(\alpha^2-\alpha-4)(3-\alpha)\alpha}{(\alpha^2-3\alpha-2)(\alpha+1)(2-\alpha)}\left(\alpha,\frac{1}{2-\alpha}\right)\right).
$$
Now $R=Q\left(\alpha,\frac{1}{2-\alpha}\right)$ amounts to
$1=\frac{(\alpha^2-\alpha-4)(3-\alpha)\alpha}{(\alpha^2-3\alpha-2)(\alpha+1)(2-\alpha)}$. Subtracting $1$, multiplying by the denominator and dividing by $2(1-\alpha)$ gives the contradiction $0=2$.

\emph{Case II. } Here $R=((L: L)\cdot L)^F:(L\cdot L)^F$ gives
$$
R=Q\left(\frac{(4-\alpha)(3-\alpha)\alpha}{(\alpha+2)(\alpha+1)(2-\alpha)}\left(\alpha,\frac{1}{2-\alpha}\right)\right)
$$
as in Case I.c.
\end{proof}

\begin{proof}[Proof of Theorem~\ref{thm:main}(iv) and (v)]
Clearly, every trapezoid is $n$-gc-self-affine for all $n \ge 2$, see the rightmost illustration in Figure~\ref{fig:1}. The positive claims of Theorem~\ref{thm:main}(iv) and (v) for non-trapezoids are given by Lemma~\ref{lem:odd_positive}. The remaining negative result for kites in Theorem~\ref{thm:main}(v) is the claim of Lemma~\ref{lem:kites5}.
\end{proof}


\section{$3$-gc-self-affine quadrangles\label{sec:n=3}}

\begin{thm}\label{thm:n=3}
A convex quadrangle is $3$-gc-self-affine if and only if it belongs to one of the following families.
\begin{itemize}
\item[(I)]
$\{T(\alpha):\,0 < \alpha < 1\} \cup \{P\}$ (the family of all trapezoids including parallelograms).

\item[(II)]
$\left\{Q\left(\alpha,\frac{-1+\sqrt{1+4\alpha-4\alpha^2}}{2\alpha(1-\alpha)}\right):0 < \alpha < 1\right\}$.\\
Here $\alpha < \frac{-1+\sqrt{1+4\alpha-4\alpha^2}}{2\alpha(1-\alpha)} < 1$ whenever $0 < \alpha < 1$.

\item[(III)]
$\left\{Q\left(\alpha,\frac{1-3\alpha+\alpha^2+\sqrt{1-2\alpha+7\alpha^2-6\alpha^3+\alpha^4}}{2(1-\alpha)}\right):0 < \alpha < 1\right\}$.\\
Here $\alpha < \frac{1-3\alpha+\alpha^2+\sqrt{1-2\alpha+7\alpha^2-6\alpha^3+\alpha^4}}{2(1-\alpha)} < 1$ whenever $0 < \alpha < 1$.

\item[(IV)]
$\Big\{Q(\alpha,\beta):0 < \alpha < \beta < 1,$\\
$\phantom{\Big\{Q
}\left(\alpha-\alpha^2\right)\beta^3+\left(1-2\alpha+2\alpha^2\right)\beta^2+\left(-1+2\alpha-4\alpha^2+\alpha^3\right)\beta+\alpha^2=0 \Big\}$.\\
$\text{Here }\Big\{(\alpha,\beta): 0 < \alpha < \beta < 1,$\\
$\phantom{\text{Here }
} \left(\alpha-\alpha^2\right)\beta^3+\left(1-2\alpha+2\alpha^2\right)\beta^2+\left(-1+2\alpha-4\alpha^2+\alpha^3\right)\beta+\alpha^2=0\Big\}$ is the graph of a function $\beta:(0,1) \to (0,1)$ with $\alpha < \beta(\alpha) < 1$ whenever $0 < \alpha < 1$.
\end{itemize}
\end{thm}

\begin{proof}
We restrict our consideration to the classes $Q(\alpha,\beta)$, $0 < \alpha < \beta < 1$, because the situation is trivial for trapezoids.

\emph{Step 1. Reduction to nine extended dissection trees. }
 If $R$ is the root of a dissection tree with leafs $L=Q(\alpha,\beta)$ of a $3$-gc-self-affinity, then one child of $R$ is a leaf and the other child is a quadrangle $Q$ whose both children are leafs. By the commutativity of the operations $\cdot$ and $:$, w.l.o.g., $Q$ is the left child of $R$. Moreover, the following are satisfied for the extended dissection tree, see the left-hand part of Figure~\ref{fig:n=3}.
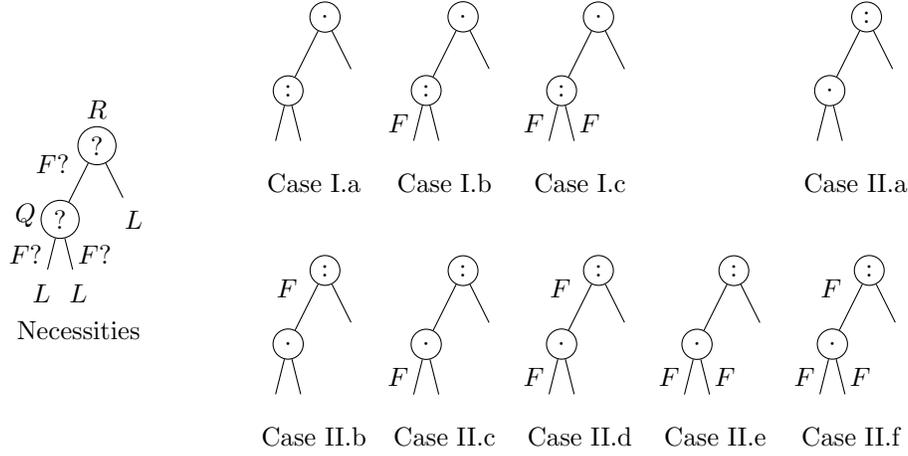
\begin{figure}
\begin{center}
\begin{tikzpicture}[scale=.45]

\draw (-2,-3.5) node{\tikz[scale=.49]{

\node[shape=circle,draw=black] (R) at (0,0) {\!\!$?$\!\!};
\node[shape=circle,draw=black] (Q) at (-1,-2) {\!\!$?$\!\!};
\node[shape=circle] (L1) at (-1.5,-4) {$\quad$};
\node[shape=circle] (L2) at (-.5,-4) {$\quad$};
\node[shape=circle] (L3) at (1,-2) {$\quad$};

\path [-] (R) edge node[above left] {$F?$} (Q);
\path [-] (R) edge node[above right] {$$} (L3);
\path [-] (Q) edge node[left] {$F?$} (L1);
\path [-] (Q) edge node[right] {$F?$} (L2);

\draw
  (0,.5) node[above] {$R$}
  (-1.4,-1.9) node[left] {$Q$}
  (L1) node {$L$}
  (L2) node {$L$}
  (L3) node {$L$}
  (-.5,-5) node {Necessities}
	;

}}
;

\draw (5,0) node{\tikz[scale=.49]{

\node[shape=circle,draw=black] (R) at (0,0) {\!\!$\cdot$\!\!};
\node[shape=circle,draw=black] (Q) at (-1,-2) {\!\!$:$\!\!};
\node[shape=circle] (L1) at (-1.5,-4) {$\quad$};
\node[shape=circle] (L2) at (-.5,-4) {$\quad$};
\node[shape=circle] (L3) at (1,-2) {$\quad$};

\path [-] (R) edge node[above left] {$$} (Q);
\path [-] (R) edge node[above right] {$$} (L3);
\path [-] (Q) edge node[left] {$$} (L1);
\path [-] (Q) edge node[right] {$$} (L2);

\draw
  (-.3,-4.5) node {Case I.a}
;

}}
;

\draw (9,0) node{\tikz[scale=.49]{

\node[shape=circle,draw=black] (R) at (0,0) {\!\!$\cdot$\!\!};
\node[shape=circle,draw=black] (Q) at (-1,-2) {\!\!$:$\!\!};
\node[shape=circle] (L1) at (-1.5,-4) {$\quad$};
\node[shape=circle] (L2) at (-.5,-4) {$\quad$};
\node[shape=circle] (L3) at (1,-2) {$\quad$};

\path [-] (R) edge node[above left] {$$} (Q);
\path [-] (R) edge node[above right] {$$} (L3);
\path [-] (Q) edge node[left] {$F$} (L1);
\path [-] (Q) edge node[right] {$$} (L2);

\draw
  (-.5,-4.5) node {Case I.b}
;

}}
;

\draw (13,0) node{\tikz[scale=.49]{

\node[shape=circle,draw=black] (R) at (0,0) {\!\!$\cdot$\!\!};
\node[shape=circle,draw=black] (Q) at (-1,-2) {\!\!$:$\!\!};
\node[shape=circle] (L1) at (-1.5,-4) {$\quad$};
\node[shape=circle] (L2) at (-.5,-4) {$\quad$};
\node[shape=circle] (L3) at (1,-2) {$\quad$};

\path [-] (R) edge node[above left] {$$} (Q);
\path [-] (R) edge node[above right] {$$} (L3);
\path [-] (Q) edge node[left] {$F$} (L1);
\path [-] (Q) edge node[right] {$F$} (L2);

\draw
  (-.5,-4.5) node {Case I.c}
;

}}
;

\draw (21,0) node{\tikz[scale=.49]{

\node[shape=circle,draw=black] (R) at (0,0) {\!\!$:$\!\!};
\node[shape=circle,draw=black] (Q) at (-1,-2) {\!\!$\cdot$\!\!};
\node[shape=circle] (L1) at (-1.5,-4) {$\quad$};
\node[shape=circle] (L2) at (-.5,-4) {$\quad$};
\node[shape=circle] (L3) at (1,-2) {$\quad$};

\path [-] (R) edge node[above left] {$$} (Q);
\path [-] (R) edge node[above right] {$$} (L3);
\path [-] (Q) edge node[left] {$$} (L1);
\path [-] (Q) edge node[right] {$$} (L2);

\draw
  (-.3,-4.5) node {Case II.a}
;

}}
;

\draw (5,-7.5) node{\tikz[scale=.49]{

\node[shape=circle,draw=black] (R) at (0,0) {\!\!$:$\!\!};
\node[shape=circle,draw=black] (Q) at (-1,-2) {\!\!$\cdot$\!\!};
\node[shape=circle] (L1) at (-1.5,-4) {$\quad$};
\node[shape=circle] (L2) at (-.5,-4) {$\quad$};
\node[shape=circle] (L3) at (1,-2) {$\quad$};

\path [-] (R) edge node[above left] {$F$} (Q);
\path [-] (R) edge node[above right] {$$} (L3);
\path [-] (Q) edge node[left] {$$} (L1);
\path [-] (Q) edge node[right] {$$} (L2);

\draw
  (-.3,-4.5) node {Case II.b}
;

}}
;

\draw (9,-7.5) node{\tikz[scale=.49]{

\node[shape=circle,draw=black] (R) at (0,0) {\!\!$:$\!\!};
\node[shape=circle,draw=black] (Q) at (-1,-2) {\!\!$\cdot$\!\!};
\node[shape=circle] (L1) at (-1.5,-4) {$\quad$};
\node[shape=circle] (L2) at (-.5,-4) {$\quad$};
\node[shape=circle] (L3) at (1,-2) {$\quad$};

\path [-] (R) edge node[above left] {$$} (Q);
\path [-] (R) edge node[above right] {$$} (L3);
\path [-] (Q) edge node[left] {$F$} (L1);
\path [-] (Q) edge node[right] {$$} (L2);

\draw
  (-.5,-4.5) node {Case II.c}
;

}}
;

\draw (13,-7.5) node{\tikz[scale=.49]{

\node[shape=circle,draw=black] (R) at (0,0) {\!\!$:$\!\!};
\node[shape=circle,draw=black] (Q) at (-1,-2) {\!\!$\cdot$\!\!};
\node[shape=circle] (L1) at (-1.5,-4) {$\quad$};
\node[shape=circle] (L2) at (-.5,-4) {$\quad$};
\node[shape=circle] (L3) at (1,-2) {$\quad$};

\path [-] (R) edge node[above left] {$F$} (Q);
\path [-] (R) edge node[above right] {$$} (L3);
\path [-] (Q) edge node[left] {$F$} (L1);
\path [-] (Q) edge node[right] {$$} (L2);

\draw
  (-.5,-4.5) node {Case II.d}
;

}}
;

\draw (17,-7.5) node{\tikz[scale=.49]{

\node[shape=circle,draw=black] (R) at (0,0) {\!\!$:$\!\!};
\node[shape=circle,draw=black] (Q) at (-1,-2) {\!\!$\cdot$\!\!};
\node[shape=circle] (L1) at (-1.5,-4) {$\quad$};
\node[shape=circle] (L2) at (-.5,-4) {$\quad$};
\node[shape=circle] (L3) at (1,-2) {$\quad$};

\path [-] (R) edge node[above left] {$$} (Q);
\path [-] (R) edge node[above right] {$$} (L3);
\path [-] (Q) edge node[left] {$F$} (L1);
\path [-] (Q) edge node[right] {$F$} (L2);

\draw
  (-.5,-4.5) node {Case II.e}
;

}}
;

\draw (21,-7.5) node{\tikz[scale=.49]{

\node[shape=circle,draw=black] (R) at (0,0) {\!\!$:$\!\!};
\node[shape=circle,draw=black] (Q) at (-1,-2) {\!\!$\cdot$\!\!};
\node[shape=circle] (L1) at (-1.5,-4) {$\quad$};
\node[shape=circle] (L2) at (-.5,-4) {$\quad$};
\node[shape=circle] (L3) at (1,-2) {$\quad$};

\path [-] (R) edge node[above left] {$F$} (Q);
\path [-] (R) edge node[above right] {$$} (L3);
\path [-] (Q) edge node[left] {$F$} (L1);
\path [-] (Q) edge node[right] {$F$} (L2);

\draw
  (-.5,-4.5) node {Case II.f}
;

}}
;

\end{tikzpicture}
\end{center}
\caption{Particular extended dissection trees for the proof of Theorem~\ref{thm:n=3}.\label{fig:n=3}}
\end{figure}

\begin{itemize}
\item[(A)]
W.l.o.g., there is no flip on the edge directly connecting $R$ with a leaf $L$.

\item[(B)] 
If there is $:$ at $Q$ then there is no flip between $Q$ and $R$.

\item[(C)] 
If there is exactly one flip below $Q$ then, w.l.o.g., it is on the left edge below $Q$.

\item[(D)]
There is $\cdot$ at $R$ if and only if there is $:$ at $Q$.
\end{itemize}

For (A), note that if there is a dissection tree with leafs $L=Q(\alpha,\beta)$, then the same dissection is described by leafs $L=Q(\alpha,\beta)^F$ when flips on edges emanating from leafs are replaced by non-flips and vice versa.

For (B), if there is $:$ at $Q$ and a flip between $Q$ and $R$, then $Q=T(\gamma)$ is a trapezoid and the operation at $R$ combines $T(\gamma)^F$ with a non-trapezoid. This is impossible. 

Claim (C) follows from the commutativity of $\cdot$ and $:$.

For (D), assume first that we have $\cdot$ at $R$ and $Q$ simultaneously. Then $R$ has the affine quotient $\big(\frac{\alpha}{\beta}\big)^3$ and cannot be of the same affine type as $L$, whose quotient is $\frac{\alpha}{\beta}$, a contradiction. If there is $:$ at both $R$ and $Q$, then $Q=T(\gamma)$ is a trapezoid. By (A) and (B), there are no flips directly below $R$. Thus $R=T(\gamma) : Q(\alpha,\beta)=Q(\alpha,\beta):T(\gamma)$, which does not make sense.

The restrictions (A), (B), (C) and (D) reduce the dissection trees to Cases I.a-c if there is $\cdot$ at $R$ (whence there is $:$ at $Q$) and to Cases II.a-f in the opposite situation, see Figure~\ref{fig:n=3}. 

\emph{Step 2. Discussion of the nine trees. }

\begin{itemize}
\item[(E)]
Any of the remaining trees with leafs $L=Q(\alpha,\beta)$
describes a $3$-gc-self-affinity of a quadrangle if and only if the parametrization of its root $R$ obtained by computation along the tree from $L=Q(\alpha,\beta)$ coincides with $Q(\alpha,\beta)^F=Q\left(\frac{1-\beta}{(1-\alpha)\beta}(\alpha,\beta)\right)$.
\end{itemize}

We have a $3$-gc-self-affinity if and only if $R$ is of the same affine type as $L=Q(\alpha,\beta)$; that is, if the parametrization of $R$ is one of $Q(\alpha,\beta)$ or $Q(\alpha,\beta)^F$. In all cases the parametrization of $R$ is obtained by either $\cdot$ or $:$ with one operand being $L=Q(\alpha,\beta)$. This way we obtain a parameter by multiplication of $\alpha$ with some positive number less than one. But then the parametrization of $R$ is not $Q(\alpha,\beta)$, because its smaller parameter is smaller than $\alpha$. Consequently, $R$ and $L$ have the same affine type if and only if the parametrization of $R$ is $Q(\alpha,\beta)^F$. 

In the following we compute the parametrization of $R$ from $L=Q(\alpha,\beta)$ for all trees determined in Step 1 and analyse criterion (E). Support by a computer algebra system is useful, but not necessary.

\emph{Case I.a. } Here $R=Q(\alpha\beta(\alpha,\beta))$ and (E) amounts to $\alpha\beta=\frac{1-\beta}{(1-\alpha)\beta}$. This is equivalent to $p_\alpha(\beta):=\beta^2+\frac{1}{\alpha(1-\alpha)}\beta-\frac{1}{\alpha(1-\alpha)}=0$. For every $\alpha \in (0,1)$, the quadratic function $p_\alpha(\beta)$ has a unique root $\beta \in (\alpha,1)$, because $p_\alpha(\alpha)=\frac{\alpha^3-1}{\alpha}<0<1=p_\alpha(1)$. That root is $\beta(\alpha)= \frac{-1+\sqrt{1+4\alpha-4\alpha^2}}{2\alpha(1-\alpha)}$, and we obtain family (II) from Theorem~\ref{thm:n=3}.

\emph{Case I.b. } Now $R=Q\left(\frac{(1-\beta)\alpha}{1-\alpha}(\alpha,\beta)\right)$ and (E) reads as $\frac{(1-\beta)\alpha}{1-\alpha}=\frac{1-\beta}{(1-\alpha)\beta}$. This is equivalent to $\alpha\beta=1$, contradicting $0 < \alpha < \beta < 1$.

\emph{Case I.c. } We get $R=Q\left(\frac{(1-\beta)^2\alpha}{(1-\alpha)^2\beta}(\alpha,\beta)\right)$ and (E) gives $\frac{(1-\beta)^2\alpha}{(1-\alpha)^2\beta}=\frac{1-\beta}{(1-\alpha)\beta}$. Multiplication with $\frac{(1-\alpha)^2\beta}{1-\beta}$ and subtraction of the right-hand side yields $(1-\beta)\alpha-(1-\alpha)=-\left((1-\alpha)^2+\alpha(\beta-\alpha)\right)=0$, again contradicting $0 < \alpha < \beta < 1$.

\emph{Case II.a. } We have $R=Q(\alpha\beta(\alpha,\beta))$ as in Case~I.a.

\emph{Case II.b. } Here $R=Q\left(\frac{(1-\beta^2)\alpha}{(1-\alpha^2)\beta}(\alpha,\beta)\right)$ and (E) is $\frac{(1-\beta^2)\alpha}{(1-\alpha^2)\beta}=\frac{1-\beta}{(1-\alpha)\beta}$. Multiplying with $\frac{(1-\alpha^2)\beta}{1-\beta}$ and subtracting $\alpha$ we obtain $\alpha\beta=1$, contradicting $0<\alpha<\beta<1$.

\emph{Case II.c. } We get $R=Q\left(\frac{(1-\beta)\alpha}{1-\alpha}(\alpha,\beta)\right)$ as in Case~I.b.

\emph{Case II.d. } Here $R=Q\left(\frac{((1-\alpha)-(1-\beta)\beta)\alpha}{(1-\alpha)\beta-(1-\beta)\alpha^2}(\alpha,\beta)\right)$ and (E) is $\frac{((1-\alpha)-(1-\beta)\beta)\alpha}{(1-\alpha)\beta-(1-\beta)\alpha^2}=\frac{1-\beta}{(1-\alpha)\beta}$. Note that the left-hand denominator is positive, because $1-\alpha>1-\beta>0$ and $\beta>\alpha^2>0$. By multiplying with the denominators and subtracting the right-hand side, (E) is equivalent to
\begin{equation}\label{eq:pfn=3}
\left(\alpha-\alpha^2\right)\beta^3+\left(1-2\alpha+2\alpha^2\right)\beta^2+\left(-1+2\alpha-4\alpha^2+\alpha^3\right)\beta+\alpha^2=0.
\end{equation}
This describes family (IV) from Theorem~\ref{thm:n=3}. To see that, for every $\alpha \in (0,1)$, equation \eqref{eq:pfn=3} has a unique solution $\beta=\beta(\alpha) \in (\alpha,1)$, we denote the left-hand side of \eqref{eq:pfn=3} by $h_\alpha(\beta)$. There is a solution, since $h_\alpha(\alpha)=-\alpha(1-\alpha)^4< 0 < \alpha(1-\alpha)^2=h_\alpha(1)$. The solution is unique, because $h_\alpha$ is convex on the interval $(\alpha,1)$ by
$h_\alpha''(\beta)=6\alpha(1-\alpha)\beta+2\left((1-\alpha)^2+\alpha^2)\right)>0$.

\emph{Case II.e. } We have $R=Q\left(\frac{(1-\beta)^2\alpha}{(1-\alpha)^2\beta}(\alpha,\beta)\right)$ as in Case~I.c.

\emph{Case II.f. } Here $R=Q\left(\frac{(2-\alpha-\beta)\alpha\beta}{(1-\alpha)\beta+(1-\beta)\alpha}(\alpha,\beta)\right)$ and (E) reads as $\frac{(2-\alpha-\beta)\alpha\beta}{(1-\alpha)\beta+(1-\beta)\alpha}=\frac{1-\beta}{(1-\alpha)\beta}$. By subtracting the right-hand side and multiplying with the denominators and the positive term $\frac{1}{(1-\alpha)(1-\alpha\beta)}$, we arrive at the equivalent formulation 
$q_\alpha(\beta):=\beta^2-\frac{1-3\alpha+\alpha^2}{1-\alpha}\beta-\frac{\alpha}{1-\alpha}=0$.
For every $\alpha \in (0,1)$, the quadratic function $q_\alpha(\beta)$ has a unique root $\beta \in (\alpha,1)$, because $q_\alpha(\alpha)=-2\alpha(1-\alpha)<0<\alpha=q_\alpha(1)$. Noting that the smaller root $\beta_1=\frac{1-3\alpha+\alpha^2 -\sqrt{(1-3\alpha+\alpha^2)^2+4\alpha(1-\alpha)}}{2(1-\alpha)}$ is negative, the above mentioned relevant root is $\beta_2=\frac{1-3\alpha+\alpha^2 +\sqrt{(1-3\alpha+\alpha^2)^2+4\alpha(1-\alpha)}}{2(1-\alpha)}$. This way we get family (III) from Theorem~\ref{thm:n=3}.
\end{proof}

Simple numerical computations show that all families in Theorem~\ref{thm:n=3} are relevant: not all quadrangles represented by one of the families (II), (III) and (IV) are contained in the other families. 


\section{An observation on self-affine convex quadrangles\label{sec:self-affine}}

\begin{proof}[Proof of Theorem~\ref{thm:self-affine}]
By Theorem~\ref{thm:main}(iv), it remains to consider the case $n=5$ and the cases with even $n \ge 6$. Moreover, the situation for trapezoids is trivial, so that we restrict our consideration to quadrangles $Q(\alpha,\beta)$, $0 < \alpha < \beta < 1$.

\emph{The case $n=5$. } This dissection can be found in \cite[Proposition 1]{hertel_richter2010} and goes back to Attila P\'or: Let $Q(\alpha,\beta)$ be represented by a quadrangle $Q=abcd$ with $a=(0,0)$ as in Figure~\ref{fig:parameter1}, and let $\varrho$ be a contraction with centre $a$ and ratio $\alpha\beta$. Then $Q$ splits into the contracted quadrangle $\varrho(Q)$ and two trapezoids $T_1=\varrho(b)bc\varrho(c)$ and $T_2=\varrho(c)cd\varrho(d)$. $T_1$ and $T_2$ are of affine type $T(\alpha\beta)$ and can be gc-dissected into two copies of $Q(\alpha,\beta)$ by Lemma~\ref{lem:trapezoids}(i). This way $Q(\alpha,\beta)$ is dissected into five affine copies of itself.

\emph{The case of even $n \ge 6$. }
Let the quadrangle $Q^*=abcd$ represent $Q(\alpha,\beta)$ as in Figure~\ref{fig:parameter1}. Now let $\varphi$ be the affine map defined by $\varphi(a)=c$, $\varphi(b)=b$ and $\varphi(d)=d$. Note that $\varphi(c)$ is between $a$ and $c$ and, in particular, $\varphi(Q^*) \subseteq Q^*$; see the left-hand part of Figure~\ref{fig:self-affine}. 
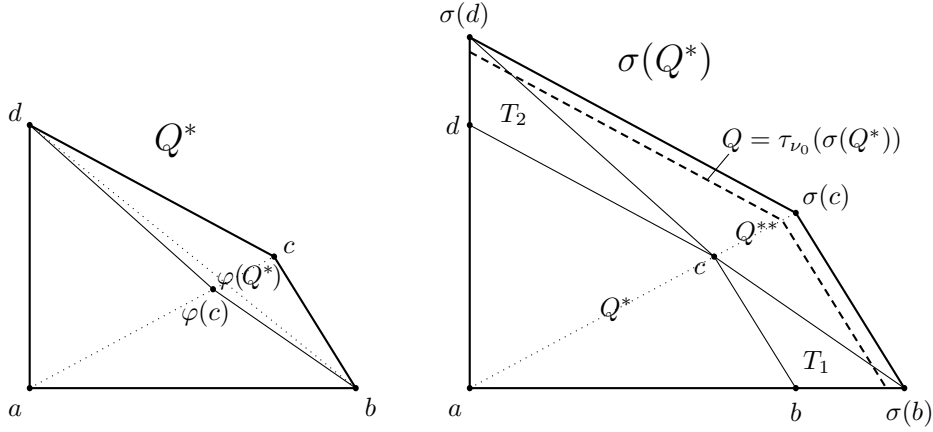
\begin{figure}
\begin{center}
\begin{tikzpicture}[xscale=6.5, yscale=7]

\draw[thick]
  (0,0)--(.889,0)--(.667,.333)--(0,.667)--cycle
	(-.9,0)--(-.233,0)--(-.4,.25)--(-.9,.5)--cycle
	;

\draw
	(0,.5)--(.5,.25)--(.667,0)
	(0,.667)--(.5,.25)--(.889,0)
	(-.9,.5)--(-.525,.188)--(-.233,0)
	(.485,.395)--(.53,.445)
  ;

\draw[thick,densely dashed]
  (0,.639)--(.639,.319)--(.852,0)
	;

\draw[dotted]
  (0,0)--(.667,.333)
	(-.9,0)--(-.4,.25)
	(-.233,0)--(-.9,.5)
	;

\fill
  (-.9,0) circle (.006)
	(-.93,-.07) node[above] {$a$}
  (-.233,0) circle (.006)
	(-.203,-.07) node[above] {$b$}
  (-.4,.25) circle (.006)
	(-.37,.24) node[above] {$c$}
  (-.9,.5) circle (.006)
	(-.93,.49) node[above] {$d$}
  (-.525,.188) circle (.006)
	(-.54,.1) node[above] {$\varphi(c)$}
	(-.45,.21) node {$\varphi(\hspace{-1pt}Q^*\hspace{-1pt})$}
	(-.6,.47) node {\LARGE $Q^*$}
	(.3,.15) node {$Q^*$}
	(.588,.294) node {$Q^{**}$}
	(.71,.05) node {$T_1$}
	(.09,.52) node {$T_2$}
  (0,0) circle (.006)
	(-.03,-.07) node[above] {$a$}
  (.667,0) circle (.006)
	(.667,-.08) node[above] {$b$}
  (.5,.25) circle (.006)
	(.47,.2) node[above] {$c$}
  (0,.5) circle (.006)
	(-.035,.5) node {$d$}
	  (.889,0) circle (.006)
	(.9,-.09) node[above] {$\sigma(b)$}
  (.667,.333) circle (.006)
	(.73,.32) node[above] {$\sigma(c)$}
  (0,.667) circle (.006)
	(-.01,.665) node[above] {$\sigma(d)$}
	(.7,.47) node {$Q=\tau_{\nu_0}(\sigma(Q^*))$}
(.4,.62) node {\LARGE $\sigma(Q^*)$}
	;

\end{tikzpicture}
\end{center}
\caption{$n$-self-affinity for even $n \ge 6$.\label{fig:self-affine}}
\end{figure}
Next let $\sigma$ be a homothety with centre $a$ such that $\sigma(\varphi(c))=c$; see the right-hand part of Figure~\ref{fig:self-affine}. Then $\sigma(Q^*)$ is dissected into the quadrangles $Q^*$ and $Q^{**}=\sigma(\varphi(Q^*))$ and two triangles $T_1=b\sigma(b)c$ and $T_2=c\sigma(d)d$. Finally, let $\tau_\nu$ denote a homothety with centre $a$ and ratio $\nu$ with $\frac{|ac|}{|a\sigma(c)|}<\nu<1$. Then the quadrangle $\tau_\nu(\sigma(Q^*))$ (dashed in Figure~\ref{fig:self-affine}) is of affine type $Q(\alpha,\beta)$ and dissected into the quadrangles $Q^*$ and $Q^{**}\cap\tau_\nu(\sigma(Q^*))$ of type $Q(\alpha,\beta)$ and into two trapezoids $T_1 \cap\tau_\nu(\sigma(Q^*))$ and $T_2 \cap\tau_\nu(\sigma(Q^*))$ of the same type $T(\mu(\nu))$. Here the parameter $\mu(\nu)$ depends continuously on $\nu$ with $\lim_{\nu \downarrow \frac{|ac|}{|a\sigma(c)|}} \mu(\nu)=1$ and $\lim_{\nu \uparrow 1}\mu(\nu)=0$. So we find $\nu_0$ such that $\mu(\nu_0)=\alpha\beta$.
Consequently, the quadrangle $Q=\tau_{\nu_0}(\sigma(Q^*))$ of type $Q(\alpha,\beta)$ splits into two quadrangles of type $Q(\alpha,\beta)$ and two trapezoids of type $T(\alpha\beta)$. By Lemma~\ref{lem:trapezoids}, one of the trapezoids can be dissected into two quadrangles of type $Q(\alpha,\beta)$ and the other one can be dissected into $n-4$ quadrangles of type $Q(\alpha,\beta)$. This proves the $n$-self-affinity of $Q$.
 \end{proof}



\end{document}